\documentclass[11pt,letterpaper,notitlepage,twoside]{article}

\usepackage[top=4cm, bottom=3.5cm, left=3.8cm, right=3.8cm, columnsep=20pt]{geometry}
\let\emph\relax 
\DeclareTextFontCommand{\emph}{\bfseries\em}

\usepackage[dvipsnames]{xcolor}
\definecolor{FrenchPink}{RGB}{255,118,164}
\definecolor{VividMulberry}{RGB}{192,17,215}
\definecolor{DenimBlue}{RGB}{47,60,190}
\usepackage{hyperref}
\hypersetup{
	colorlinks=true,     
	linkcolor=DenimBlue,    
	citecolor=FrenchPink,     
	filecolor=VividMulberry,      
	urlcolor=VividMulberry,       
}
\usepackage{enumitem}

\usepackage{titlesec}
\renewcommand{\thesection}{\Roman{section}}
\titleformat{\section}[block]{\large\scshape}{\thesection.}{0.5em}{}
\renewcommand{\thesubsection}{\Roman{section}.\alph{subsection}}
\titleformat{\subsection}[block]{\scshape}{\thesubsection.}{0.5em}{}
\titleformat{\subsubsection}[block]{\bfseries}{}{0.5em}{}

\newcommand{\pagetitre}[4]{
	\noindent\rule{\linewidth}{0.5pt}
	\begin{center}
		\begin{tabular}{c}
			{\LARGE\textmd{\scshape #1}}\\[2pt]
			{\LARGE\textmd{\scshape #2}}\\[15pt]   
			{\large #3 {\scshape #4}$^\dagger$}
		\end{tabular}
	\end{center}
	\vspace{-0.2cm}
	\noindent\rule{\linewidth}{0.5pt}}

\usepackage{fancyhdr}
\fancypagestyle{first}{
	\pagestyle{fancy}
	\fancyhead[C]{}
	\fancyhead[RO,LE]{}
	\fancyhead[LO,RE]{}
	\fancyfoot[C]{}
	\fancyfoot[L]{{\footnotesize Latest update: \today \\[1pt]
			$^\dagger$\support\\[1cm]}}
	
	}

\makeatletter
\renewenvironment{abstract}{%
	\if@twocolumn
	\section*{\abstractname}%
	\else\small 
	\begin{center}%
		{\scshape \abstractname\vspace{-0.25cm}}
	\end{center}%
	\quotation
	\fi}
{\if@twocolumn\else\endquotation\fi}
\makeatother

\let \savenumberline \numberline
\def \numberline#1{\savenumberline{#1.}}

\makeatletter
\renewcommand{\l@section}{\@dottedtocline{1}{1.5em}{2.6em}}
\makeatother

\makeatletter
\renewcommand\tableofcontents{%
	\begin{center}%
		{\scshape \contentsname\vspace{-0.25cm}}
	\end{center}%
	\@mkboth{\MakeUppercase\contentsname}{\MakeUppercase\contentsname}%
	\@starttoc{toc}%
}
\makeatother

\makeatletter
\def\namedlabel#1#2{\begingroup
	\def\@currentlabel{#2}%
	\label{#1}\endgroup
}
\makeatother

\usepackage{amsmath,tabu}
\usepackage{amsthm}
\usepackage{amssymb}
\usepackage{mathrsfs}
\usepackage{mathtools}
\usepackage{dsfont}
\usepackage{etoolbox}
\newcommand{\rom}[1]{\uppercase\expandafter{\romannumeral #1\relax}}

\usepackage{tabularx}
\usepackage{graphicx}
\usepackage{caption}
\usepackage{subcaption}
\usepackage{float}
\usepackage{booktabs}
\usepackage{wrapfig}
\usepackage{tikz}
\usepackage{tikz-cd}

\newtheoremstyle{lem}{}{}{\slshape}{}{\bfseries}{}{.5em}{}
\theoremstyle{lem}
\newtheorem{lem}{Lemma}
\newtheorem{prop}{Proposition}
\newtheorem{cor}{Corollary}
\newtheorem{thm}{Theorem}
\newtheorem*{defn*}{Definition}

\newtheorem{thmout}{Theorem} 

\newtheorem*{cor*}{Corollary}

\AddToHook{cmd/appendix/before}{%
	\setcounter{equation}{0}%
	\setcounter{lem}{0}%
}

\newtheoremstyle{ex}{}{}{}{}{\scshape}{{:}}{.5em}{}
\theoremstyle{ex}
\newtheorem*{ex*}{Example}
\newtheorem*{ack*}{Acknowledgements}

\newtheoremstyle{rem}{}{}{\slshape}{}{\itshape}{.}{.5em}{}
\theoremstyle{rem}
\newtheorem{rem}{Remarks}

\newtheoremstyle{pr}{}{}{}{}{\scshape}{.}{.5em}{}
\theoremstyle{pr}
\newtheorem*{pr*}{Proof}
\AtEndEnvironment{pr*}{\null\hfill\qedsymbol}

\renewenvironment{proof}[1][\proofname]{{\noindent\scshape #1:}}{\null\hfill\qedsymbol}

\usepackage{newpxtext,newpxmath}
\usepackage[utf8]{inputenc}
\usepackage[T1]{fontenc}

\usepackage{lipsum}

\newcommand{\titre}{Lagrangian metric geometry} 
\newcommand{\titrep}{with Riemannian bounds} 
\newcommand{\titrepp}{Lagrangian metric geometry with Riemannian bounds}
\newcommand{\prenomauteur}{Jean-Philippe} 
\newcommand{\nomauteur}{Chass\'e} 
\newcommand{\support}{The author is partially supported by the Swiss National Science Foundation (200021\_204107)}

\newcommand{\Addresses}{{
		\bigskip
		\footnotesize
		
		\noindent J.-P.~Chass\'{e}, \textsc{Department of Mathematics,  ETH Z\"urich}\par\nopagebreak
		\noindent \textit{E-mail address}: \texttt{jeanphilippe.chasse@math.ethz.ch}
		
}}

\newcommand{\N}{\mathds{N}}
\newcommand{\Z}{\mathds{Z}}

\newcommand{\R}{\mathds{R}}
\newcommand{\C}{\mathds{C}}

\newcommand{\Id}{\mathds{1}}

\newcommand{\Ham}{\operatorname{Ham}}
\newcommand{\Symp}{\operatorname{Symp}}
\newcommand{\graph}{\operatorname{graph}}

\newcommand{\del}{\partial}

\let\phi\varphi
\let\epsilon\varepsilon
\let\emptyset\varnothing

\begin{document}
\thispagestyle{first}

\pagetitre{\titre}{\titrep}{\prenomauteur}{\nomauteur}

\begin{abstract}
	\noindent We study collections of exact Lagrangian submanifolds respecting some uniform Riemannian bounds, which we equip with a metric naturally arising in symplectic topology (e.g. the Lagrangian Hofer metric or the spectral metric). We exhibit many metric and symplectic properties of these spaces, such that they have compact completions and that they contain only finitely many Hamiltonian isotopy classes. We then use this to exclude many unusual phenomena from happening in these bounded spaces. Taking limits in the bounds, we also conclude that there are at most countably many Hamiltonian isotopy classes of exact Lagrangian submanifolds in a Liouville manifold. Under some mild topological assumptions, we get analogous results for monotone Lagrangian submanifolds with a fixed monotonicity constant. Finally, in the process of showing these results, we get new results on the Riemannian geometry of cotangent bundles and surfaces which might be of independent interest.
\end{abstract}

\noindent\rule{\linewidth}{0.5pt}
\tableofcontents
\noindent\rule{\linewidth}{0.5pt}

\section{Introduction and main results} \label{sec:intro}
In previous work~\cite{Chasse2023,Chasse2022}, the author showed that all the familiar metrics $d$ between Lagrangian submanifolds in symplectic topology behave like the classical Hausdorff metric when restricted to the space $\mathscr{L}_k^e$ of $\lambda$-exact Lagrangian submanifolds of a Liouville manifold $(M,\omega=d\lambda)$ which are ``geometrically bounded by k''. Broadly speaking, being geometrically bounded by $k$ ensures that all Lagrangian submanifolds considered are contained in the same compact and have curvature and volume uniformly bounded~---~we give the precise definition in Subsection~\ref{subsec:definitions}. In this paper, we continue the study of the metric spaces $(\mathscr{L}^e_k,d)$. Most notably, we will be concerned with issues of compacity and local connectedness. \par

Furthermore, we extend our results to the analogous space $\mathscr{L}^{m(\rho)}_k$ of $\rho$-monotone Lagrangian submanifolds which ``bounds enough disks''~---~the ambient symplectic manifold $M$ is either closed or convex at infinity in that case. We will not give right away the precise definition of what we mean here by ``bounding enough disks'', but we note that this is automatically satisfied if either the Lagrangian submanifolds or the symplectic manifold we consider are simply connected. \par

In what follows, $\mathscr{L}^\star_k$ will thus denote either the spaces $\mathscr{L}^e_k$ or $\mathscr{L}^{m(\rho)}_k$. We will also denote by $\mathscr{L}^\star_\infty$ the space of \textit{all} Lagrangian submanifolds respecting $\star$. Moreover, the metric $d$ that we consider will be a so-called Chekanov-type metric~\cite{Chasse2023} and be bounded from above by the Lagrangian Hofer metric. In practice, one can think of $d$ as one of the following metrics.
\begin{enumerate}[label=(\alph*)]
	\item ($d=d_H$): This is the case of the Lagrangian Hofer metric, which is due to Chekanov~\cite{Chekanov2000}. \par
	
	\item ($d=\gamma$): This is the case of the spectral metric, originally due to Viterbo~\cite{Viterbo1992} for Lagrangian submanifolds of $T^*L$ Hamiltonian isotopic to the zero-section. The metric may be extended to all exact Lagrangian submanifolds of $T^*L$ by work from Fukaya, Seidel, and Smith~\cite{FukayaSeidelSmith2008i, FukayaSeidelSmith2008ii}, Abouzaid~\cite{Abouzaid2012}, and Kragh~\cite{Kragh2013}. In general, it has also been defined for weakly exact Lagrangian submanifolds by Leclercq~\cite{Leclercq2008} and monotone ones with nonvanishing quantum homology by Kislev and Shelukhin~\cite{KislevShelukhin2022}, following work of Leclercq and Zapolsky~\cite{LeclercqZapolsky2018}. \par
	
	\item ($d=\gamma_\mathrm{ext}$): This is a variant of the usual spectral metric, as defined by Kislev and Shelukhin~\cite{KislevShelukhin2022}. \par
	
	\item ($d=\hat{d}^{\mathscr{F},\mathscr{F'}}_\mathcal{S}$): These are the shadow metrics appearing in work of Biran, Cornea, and Shelukhin~\cite{CorneaShelukhin2019, BiranCorneaShelukhin2021}. \par
	
	\item ($d=\hat{D}^{\mathscr{F},\mathscr{F'}}$): There are possibly many other weighted fragmentation pseudometrics ---~as defined by Biran, Cornea, and Zhang~\cite{BiranCorneaZhang2021}~--- that belong to this class. \par
\end{enumerate}
If $d=\gamma$ and $\star=m(\rho)$, we make a slight abuse of notation and still denote by $\mathscr{L}^{m(\rho)}_k$ the space of all $\rho$-monotone $k$-geometrically bounded Lagrangian submanifolds which both bounds enough disks and have nonvanishing quantum homology. Otherwise, we also take the convention that whenever $d$ is not properly defined between $L$ and $L'$, then $d(L,L')=+\infty$. \par

With this notation settled down, we enunciate the main principles that summarize our results.
\begin{enumerate}[label=(\arabic*)]
	\item In $(\mathscr{L}^\star_k,d)$, nothing symplectically unusual happens.
	\item Metric properties of $(\mathscr{L}^\star_k,d)$ induce topological properties on the limit space $(\mathscr{L}^\star_\infty,d)$.
\end{enumerate}
We now explain our main results and how their corollaries showcase these general ideas. Note that some of these corollaries are not as direct as others; we will properly prove all of them later in the paper. \par

\begin{thmout} \label{thmout:compact-comp}
	On $\mathscr{L}^\star_k$, all possible choices of $d$ in the above list induce the same topology and have homeomorphic completions $\widehat{\mathscr{L}}^\star_k$. Moreover, that completion $\widehat{\mathscr{L}}^\star_k$ is compact.
\end{thmout}

In term of the first principle, we get the following. \par

\begin{cor*}
	The subspaces $\mathscr{L}^{L_0}_k:=(\Ham(M)\cdot L_0)\cap\mathscr{L}^\star_k$, where $L_0\in\mathscr{L}^\star_\infty$, have finite diameter in $d$. If $M=T^*L$ and $d=\gamma$, then the same holds on $\mathscr{L}^{e}_k$.
\end{cor*}

Note that there are many known cases where it is known that $\Ham(M)\cdot L_0=\mathscr{L}^{L_0}_\infty$ has infinite diameter in $d$, e.g.\ $\mathscr{L}^{S^1}_\infty(T^*S^1)=\mathscr{L}^{e}_\infty(T^*S^1)$ has infinite Hofer diameter~\cite{Khanevsky2009}. On the other hand, it is conjectured~---~and has been proven in a many cases~\cite{Shelukhin2018, Shelukhin2022, Viterbo2022ii, GuillermouVichery2022}~---~that $\mathscr{L}^{e}_\infty(T^*Q)$ has finite diameter in the spectral metric. Therefore, without any Riemannian bound, the finiteness is highly dependent on $L_0$, $M$, and $d$. Note that we previously proved such a boundedness result~\cite{Chasse2022} on a neighbourhood of some $L$ in $\mathscr{L}^{e}_k$. The improvement here is thus in going from a neighbourhood to the whole space. \par

As we shall see below, we also manage to extend our study to spaces of graphs of Hamiltonian diffeomorphisms of closed monotone manifolds, whether these graphs bound enough disks or not. Through this, we can rule out something like Ostrover's example~\cite{Ostrover2003} from happening in the corresponding $\mathscr{L}_k$ spaces. \par

\begin{cor*}
	Suppose that $M$ is closed and monotone. On the subspace of Hamiltonian diffeomorphisms $\phi\in\Ham(M)$ whose graph has geometry bounded by $k$ in $M\times M$, the Hofer norm $||\cdot||_H$ is bounded. Moreover, the Hofer norm on diffeomorphisms and the Lagrangian Hofer distance between their graphs induce the same topology on that space.
\end{cor*}

In term of the second principle, we get the following result from Theorem~\ref{thmout:compact-comp}.

\begin{cor*}
	The limit space $(\mathscr{L}^\star_\infty,d)$ is separable, i.e. it admits a countable dense subset.
\end{cor*}

Following separate discussions with Humili\`ere and with Shelukhin, it seems that this result was well-accepted folklore on $\mathscr{L}^{L_0}_\infty=\Ham(M)\cdot L_0$. Therefore, the innovation of the above corollary seems to be on the number of Hamiltonian isotopy classes of $\mathscr{L}^\star_\infty$; this is precisely what we explore below. \par

\begin{thmout} \label{thmout:finite_classes}
	The space $\mathscr{L}^\star_k$ contains only finitely many Hamiltonian isotopy classes. Furthermore, there is an $A=A(k)>0$ such that 
	\begin{align*}
		d(L,L')\geq A
	\end{align*}
	whenever $L, L'\subseteq \mathscr{L}^\star_k$ are \emph{not} Hamiltonian isotopic.
\end{thmout}

Obviously, the $A$-bound is trivial when $d=d_H$~---~and potentially when $d=\gamma$~---~as we have taken the convention that $d_H(L,L')=\infty$ whenever $L$ and $L'$ are not Hamiltonian isotopic. However, there are examples of Lagrangian submanifolds which are not Hamiltonian isotopic but are a finite distance apart in a shadow metric~\cite{BiranCorneaShelukhin2021}. \par

We note that Theorem~\ref{thmout:finite_classes} already fits within the motif of the first principle. However, we can go even further in this direction, as the corollaries below show. \par

First, we get a result on the vanishing of entropy of symplectomorphisms preserving $\mathscr{L}^\star_k$ in some form. We refer the reader to Subsection~\ref{subsec:entropy} and the references therein for the definition of barcode and categorical entropy.

\begin{cor*}
	Let $\psi$ be a symplectomorphism of $M$ such that $\psi(\mathscr{L}^\star_\infty)=\mathscr{L}^\star_\infty$. If $L$ is such that the sequence $\{\psi^\nu(L)\}$ is fully contained in some $\mathscr{L}^\star_k$, then there is some $N$ such that $\psi^N(L)$ is Hamiltonian isotopic to $L$. Furthermore, for such $\psi$ and $L$, the barcode entropy $\hbar(\psi;L,L')$ vanishes for any $L'\in\mathscr{L}^\star_\infty$. \par
	
	More generally, if the Lagrangian submanifolds $L_1,\dots, L_\ell$ split-generate the derived Fukaya category $\mathrm{DFuk}^\star(M)$ of $M$ and each sequence $\{\psi^\nu(L_i)\}$ are contained in a single $\mathscr{L}^\star_k$, then the categorical entropy $h_\mathrm{cat}(\psi)$ of $\psi$ vanishes.
\end{cor*}

The latter part on categorical entropy is maybe more telling when formulated as its contraposition: if $h_\mathrm{cat}(\psi)>0$, then there is some Lagrangian submanifold $L$ which is a factor of a split-generator $G$ of $\mathrm{DFuk}^\star(M)$ and such that the sequence $\{\psi^\nu(L)\}$ is not contained in any $\mathscr{L}^\star_k$. Note that we may suppose that such a Lagrangian submanifold $L$ induces a nontrivial object in $\mathrm{DFuk}^\star(M)$. Therefore, this means that symplectomorphisms which are symplectically nontrivial must also be geometrically nontrivial. \par

Secondly, we get a statement that relates the path-connected component of the completions $\widehat{\mathscr{L}}_k^\star$ to the Hamiltonian isotopy classes of $\mathscr{L}^\star_\infty$. \par

\begin{cor*}
	Let $L,L'\in\mathscr{L}^\star_\infty$. Suppose that there exists a $d$-continuous path $t\mapsto L_t$ from $L$ to $L'$ in some $\widehat{\mathscr{L}}_k^\star$. Then, $L$ is Hamiltonian isotopic to $L'$.
\end{cor*}

We contrast this with a recent result from Arnaud, Humili\'ere, and Viterbo~\cite{ArnaudHumiliereViterbo2024}, which shows that such a path always exists in the completion of $\mathscr{L}^e_\infty(T^*N)$ if $d=\gamma$. Therefore, our result can be seen as very slightly reducing the gap between that result and the nearby Lagrangian conjecture. \par

As for the second principle, we get the following new result. \par

\begin{cor*}
	There are at most countably many Hamiltonian isotopy classes in $\mathscr{L}^\star_\infty$.
\end{cor*}

Note that a Liouville manifold can have infinitely many Hamiltonian isotopy classes of exact Lagrangian submanifolds. For example, if $M$ is the plumbing $T^*S^1\# T^*S^1$ of two copies of $T^*S^1$, $L$ is the zero-section in the first copy of $T^*S^1$, and $\tau:M\to M$ is the Dehn twist along the second copy $L'$, then $\tau^\nu(L)$ is clearly in a different Hamiltonian isotopy class for each $\nu\geq 0$. \par

\begin{figure}[ht]
	\centering
	\begin{tikzpicture}
		\node[anchor=south west,inner sep=0] (image) at (0,0) {\includegraphics[width=0.5\textwidth]{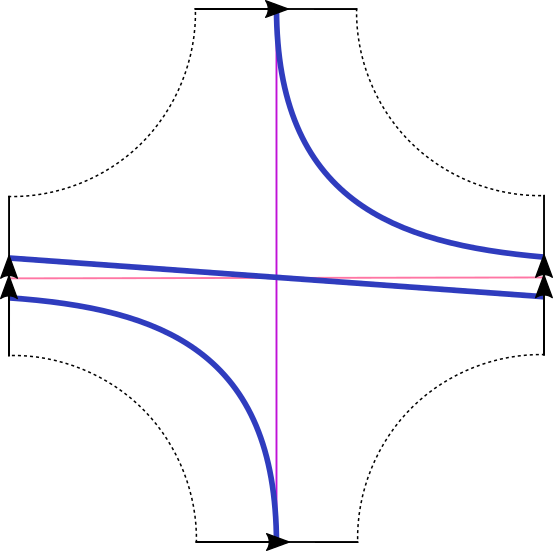}};
		\begin{scope}[x={(image.south east)},y={(image.north west)}]
			\node[Mulberry] at (0.5,1.04) {$L$};
			\node[FrenchPink] at (-0.05,0.5) {$L'$};
			\node[DenimBlue] at (1.08,0.5) {$\mathbf{\tau^2(L)}$};
		\end{scope}
	\end{tikzpicture}
	\vspace*{8pt}
	\caption{The circle $L$ (thin, purple) and its image under two Dehn twists (thick, blue) along $L'$ (thin, pink) inside some Liouville subdomain of $T^*S^1\# T^*S^1$.\label{fig:plumbing}}
\end{figure}
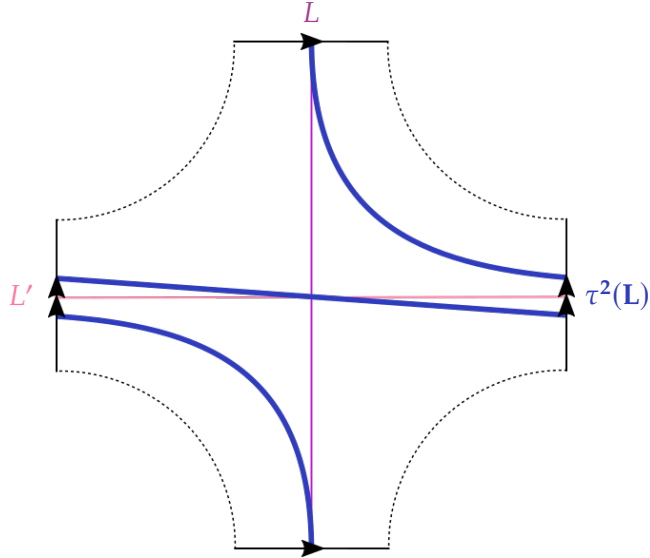

Likewise, given $\rho>0$, there are infinitely many Hamiltonian isotopy classes of $\rho$-monotone tori in $\C^3$~\cite{Auroux2015}. These automatically bound enough disks since $\pi_1(\C^3)=0$. Therefore, as our approach does not perceive the topology of $M$, we cannot expect a better bound. \par

\begin{rem} \label{rem:GPS}
	The countability of the number of exact isotopy classes of exact Lagrangian submannifolds in a Liouville sector appears as an hypothesis in the construction by Ganatra, Pardon, and Shende~\cite{GanatraPardonShende2020} of the wrapped Fukaya category. They work however with cylindrical exact Lagrangian submanifolds, which are not in general closed. That being said, we expect the techniques here to adapt by working in a fixed sector, with exact Lagrangian submanifolds which extend to cylindrical ones in the completion, and with Riemannian metrics which have some standard form near the boundary. This implies that the wrapped Fukaya category of a Liouville sector can always be constructed in such a manner as to include all exact isotopy classes.
\end{rem}

We also get a result on the local structure of the $\mathscr{L}^\star_k$ spaces in some cases. \par

\begin{thmout} \label{thmout:local_cont}
	Given an exact or monotone Lagrangian submanifold $L$ of $M$, there is a Riemannian metric $g$ on $M$ making $L$ totally geodesic and such that, for every $k>0$, it admits a system of contractible neighbourhoods in $\mathscr{L}^\star_k$. If $\dim M=2$, the latter part holds for every metric and every $k>0$ such that $L\in\mathscr{L}^\star_k$.
\end{thmout}

Though this result does not yield as many direct symplectic applications as the ones before, we can still use it to estimate the Hofer and spectral distances in some cases. More precisely, it is known~\cite{Milinkovic2001} that, for graphs of exact 1-forms in $T^*L$, the Hofer and spectral distances agree and are given by
\begin{align*}
	d_H(\graph df,\graph dg)=\gamma (\graph df,\graph dg)=\max |f-g|-\min |f-g|.
\end{align*}
In particular, for every $f\in C^\infty(L)$, $t\mapsto \graph(tdf)$ is a minimal geodesic in these metrics. However, when one embeds a neighbourhood of $L$ in $T^*L$ in $M$ via a Weinstein neighbourhood, there could be a shorter path going through $M$. Nonetheless, through Theorem~\ref{thmout:local_cont} and some estimates on Hausdorff-geodesics, we get the following estimate on how far from a minimal geodesic $t\mapsto \graph(tdf)$ can be. \par

\begin{cor*}
	Let $\Psi:D^*_r L\to M$ be a Weinstein neighbourhood of some $L\in\mathscr{L}^\star_\infty(M)$. Suppose either that the Riemannian metric $g$ on $M$ is as in Theorem~\ref{thmout:local_cont} or that $\dim M=2$. For every $k\geq 1$, there are constants $C>0$ and $r'\in(0,r]$ with the following property. Whenever $f:L\to\R$ is such that $\Psi(\operatorname{graph} df)\in\mathscr{L}^\star_k$ and $|df|\leq r'$, we have that
	\begin{align*}
		d(\Psi(t\operatorname{graph} df),\Psi(s\operatorname{graph} df))\geq C (t-s)^2 \max|df|^2
	\end{align*}
	for every $t,s\in [0,1]$.
\end{cor*}

\subsection*{Organization of the paper}
The rest of the paper is divided into four main parts and an appendix. \par

In Section~\ref{sec:setup}, we first define the objects that we will be working with in the paper and set down some notation. More precisely, we define what is meant by ``geometrically bounded by $k$'', ``bounding enough disks'', and ``being of Chekanov type''. We then move on to recall some prior results of the author which are central in proving the main results. Finally, we explain how these results cover the case of monotone Lagrangian submanifolds bounding enough disks and of monotone graphs, as these cases were not directly covered in the prior papers. \par

In Section~\ref{sec:topo}, we move to prove (pre)compactness and local contractibility of the $\mathscr{L}^\star_k$ spaces, that is, Theorems~\ref{thmout:compact-comp} and~\ref{thmout:local_cont}. These correspond, respectively, to Subsections~\ref{subsec:cpt} and~\ref{subsec:connect}. We end this part with Subsection~\ref{subsec:geodesics}, which aims to understand local Hausdorff-geodesics in $\mathscr{L}^\star_k$~---~this will be an essential step in proving our local Hofer/spectral estimate above. This also allows us to consider some variations of the Hausdorff metric on $\mathscr{L}^\star_k$ and to conclude that they also induce the same topology on these spaces. \par

In Section~\ref{sec:symp}, we finally prove symplectic properties of the $\mathscr{L}^\star_k$ spaces. More precisely, we prove Theorem~\ref{thmout:finite_classes} and all corollaries following the first principle appearing in the introduction. \par

In Section~\ref{sec:limit}, we study the limit space $\mathscr{L}^\star_\infty$ and prove the corollaries following the second principle in the introduction. We conclude that section with an analysis of an alternate limit space: $\varinjlim\mathscr{L}^\star_k$, the inductive limit of the $\mathscr{L}^\star_k$. This highlights the shortcomings of trying to study $\mathscr{L}^\star_\infty$ through the $\mathscr{L}^\star_k$ spaces. \par

Finally, we conclude the paper with an appendix which includes all the results in Riemannian geometry that we will need throughout the paper. As these results do not seem to have appeared in the literature before~---~and Lemma~\ref{lem:complete_sasaki} even appeared as a conjecture~---~we believe that they could be of independent interest. This is also why they have been compiled as an independent appendix. More precisely, Subsection~\ref{subsec:geo_sasaki} covers the new results on the Sasaki metric on $TN$, Subsection~\ref{subsec:geo_2d} those on Riemann surfaces, and Subsection~\ref{subsec:geo_comparison} those on comparison of Riemannian invariants of Lagrangian submanifolds. \par

\subsection*{Acknowledgements}
I would like to express gratitude to Octav Cornea for pushing me to write this paper. I would also like to thank both him and Paul Biran for uncountably many interesting discussions during the research process that led to this paper. Finally, I would like to thank Jean-Fran\c{c}ois Barraud for some interesting exchanges on entropy and the nearby Lagrangian conjecture, and Daniil Mamaev for pointing out to me the countability assumption in the construction of the wrapped Fukaya category. \par

\section{Preliminaries} \label{sec:setup}
We now lay down the foundations upon which the rest of the paper will be built. More precisely, in Subsection~\ref{subsec:definitions}, we give the main definitions that will be used throughout the paper. In Subsection~\ref{subsec:previous_results}, we enunciate previous results from the author that will be essential for the rest of the paper. We hope that this will improve the readability of this paper, as the notation here is slightly different than in \cite{Chasse2023} and \cite{Chasse2022}. We close things with Subsection~\ref{subsec:apply_thm2}, where it is shown that both monotone Lagrangian submanifolds bounding enough disks and monotone graphs fit within the formalism of the prior results. \par

\subsection{Definitions and notation} \label{subsec:definitions}
In this paper, we will be working with manifolds which are sometimes not smooth or not closed, i.e.\ they might have a boundary or be noncompact. However, whenever we work with such manifolds, we will make it explicitly clear. That is, when there are no mentions of it, the manifolds are assumed to be smooth, connected, and closed. \par

Let $(M,\omega)$ be either a convex-at-infinity (noncompact) or a closed symplectic manifold, and let $J$ be an $\omega$-compatible almost complex structure. When $M$ is noncompact, we assume that $J$ is convex at infinity and making $g_J=\omega(\cdot,J\cdot)$ complete. We also pick an exhaustion by compacts $W_1\subsetneq W_2\subsetneq\dots$ of $M$. When $M$ is compact, we take the convention $W_k=M$ for all $k$. \par 

In order to precisely define what was meant in the introduction by ``geometrically bounded by $k$'', we first need two definitions from Riemannian geometry. \par

\begin{defn*}
	Let $L$ be a submanifold of a Riemannian manifold $(M,g)$. Its \emph{second fundamental form} $B_L$ is given at $x\in L$ by
	\begin{center}
		\begin{tikzcd}[row sep=0pt,column sep=1pc]
			(B_L)_x\colon T_xL\otimes T_xL\otimes T_xL^\perp \arrow{r} & \R \\
			{\hphantom{(B_L)_x\colon{}}} (X,Y,N) \arrow[mapsto]{r} & {g\big(\nabla_X Y,N\big)},
		\end{tikzcd}
	\end{center}
	where $\perp$ denotes the orthogonal complement in $T_x M$ with respect to $g$ and $\nabla$ is the Levi-Civita connection $(M,g)$. We thus define the \emph{norm of the second fundamental form} to be
	\begin{align*}
		||B_L||:=\sup_{x\in L} |(B_L)_x|^{op}.
	\end{align*}
\end{defn*}

\begin{defn*}
	Let $(M,g)$ be a Riemannian manifold, and let $L$ be a submanifold. Let $\epsilon\in(0,1]$. We say that $L$ is \emph{strictly $\epsilon$-tame} if
	\begin{align*}
		\inf_{x\neq y}\frac{d_M(x,y)}{\min\{1,d_L(x,y)\}}>\epsilon,
	\end{align*}
	where $d_M$ is the distance function on $M$ induced by $g$, and $d_L$ is the distance function on $L$ induced by $g|_L$.
\end{defn*}

Then, for every $k\in\N$, we set
\begin{align*}
	\mathscr{L}_k:=\Big\{\text{Lagrangians }L\subseteq \mathrm{Int}(W_k)\ \Big|\ ||B_L||< k,\ L\text{ is strictly $(k+1)^{-1}$-tame}\Big\}.
\end{align*}
In general, this space is too big for our approach, and thus we will instead be studying the following subspaces. If $M$ is exact with $\omega=d\lambda$, we can consider
\begin{align*}
\mathscr{L}^e_k
&:=\Big\{L\in\mathscr{L}_k\ \Big|\ L\text{ is $\lambda$-exact}\Big\}
\intertext{and, if $M$ is monotone,}
\mathscr{L}^{m(\rho)}_k
&:=\Big\{L\in\mathscr{L}_k\ \Big|\ L\text{ is $\rho$-monotone \& bounds enough disks}\ \Big\},
\end{align*}
where $\rho\geq 0$. By $\rho$-monotone, we mean that $\omega=\rho\mu$ on $\pi_2(M,L)$, where $\mu$ is the Maslov index, and that the minimal Maslov number of $L$ is at least 2. Finally, we introduce the class of Lagrangian submanifolds bounding enough disks as follows.

\begin{defn*}
	We say that a Lagrangian submanifold $L$ of $M$ \emph{bounds enough disks} if the image of the composition
	\begin{center}
		\begin{tikzcd}
			\pi_2(M,L) \arrow[r,"\del"] & \pi_1(L) \arrow[r,"h"] & H_1(L;\Z)^\mathrm{free}
		\end{tikzcd}
	\end{center}
	has finite cokernel. Here, $H_1(L;\Z)^\mathrm{free}$ is the free part of $H_1(L;\Z)$, and $h$ is the composition of the abelianization homomorphism with the natural projection of $H_1(L;\Z)$ onto $H_1(L;\Z)^\mathrm{free}$.
\end{defn*}

\begin{rem} \label{rem:bounds-enough_cohomology}
	Using the fact that $H^1(L;\R)=\mathrm{Hom}(\pi_1(L),\R)=\mathrm{Hom}(H_1(L;\Z)^\mathrm{free},\R)$, we see that $L$ bounds enough disks if and only if $\del^*:H^1(L;\R)\to \mathrm{Hom}(\pi_2(M,L),\R)$ is injective. In particular, this condition is automatically satisfied if either $H^1(L;\R)=0$ or $H^1(M;\R)=0$.
\end{rem}

\begin{ex*}
	Every contractible loop in an oriented closed surface bounds enough disks. On the other hand, graphs of closed 1-form in $T^*N$ bound enough disks if and only if $H_1(N;\R)=0$.
\end{ex*}

Moreover, it will later be of interest to study graphs of Hamiltonian diffeomorphisms. Therefore, when the symplectic manifold is a product $(M\times M,-\omega\oplus\omega)$, where $(M,\omega)$ is closed and monotone, we define
\begin{align*}
	\mathscr{L}^\Gamma_k:=\left\{L\in\mathscr{L}_k(M\times M)\ \middle|\ L=\graph\phi,\ \phi\in\Ham(M,\omega)\right\}.
\end{align*}
If a result is applicable to $\mathscr{L}^e_k$, $\mathscr{L}^{m(\rho)}_k$, and $\mathscr{L}^\Gamma_k$, we will say that is true for $\mathscr{L}^{\star}_k$. We will denote the space without Riemannian bounds by $\mathscr{L}^\star_\infty$, i.e.\ $\mathscr{L}^\star_\infty=\cup_k \mathscr{L}^\star_k$. We will also sometimes write $\mathscr{L}^\star_k(M)$ when we want to make the ambient manifold $M$ more apparent. \par

Finally, we need to give a precise definition of the family of metrics in which $d$ is allowed to be. Explicitly, we will ask that $d$ be of Chekanov type and dominated by the Lagrangian Hofer metric. The former notion was defined in the author's previous work~\cite{Chasse2022}, but we recall here the definition. We however refer to said work for the proof that the metrics enunciated in the introduction are indeed of Chekanov type. \par

\begin{defn*}
	Let $\mathscr{F}\subseteq \mathscr{L}^\star(M)$. We say that a pseudometric $d^\mathscr{F}$ on $\mathscr{L}^\star(M)$ is of \emph{Chekanov type} if for all compatible almost complex structure $J$, all $\delta>0$, and all $L, L'\in\mathscr{L}^\star(M)$, there exist Lagrangian submanifolds $F_1,\dots,F_k\in\mathscr{F}$ with the following property. \par
	
	For any $C^0$- and Hofer-small Hamiltonian perturbations $\widetilde{L},\widetilde{L}',\widetilde{F}_1,\dots,\widetilde{F}_k$ of the Lagrangian submanifolds above making them pairwise transverse, and for any $x\in \widetilde{L}\cup\widetilde{L}'$, there exists a nonconstant $J$-holomorphic polygon $u:S_r\to M$ such that
	\begin{enumerate}[label=(\roman*)]
		\item has boundary along $\widetilde{L}$, $\widetilde{L}'$ and $\widetilde{F}_1,\dots,\widetilde{F}_k$; \par
		\item passes through $x$; \par
		\item respects the bound
		\begin{align*}
			\omega(u)\leq d^\mathscr{F}(L,L')+\delta.
		\end{align*}
	\end{enumerate}
	
	Let $\mathscr{F}'\subseteq \mathscr{L}^\star(M)$ be such that
	\begin{align*}
		\left(\overline{\bigcup_{F\in\mathscr{F}}F}\right)\bigcap \left(\overline{\bigcup_{F'\in\mathscr{F'}}F'}\right)
	\end{align*}
	is discrete. We will call $\hat{d}^{\mathscr{F},\mathscr{F}'}:=\max\{d^\mathscr{F},d^\mathscr{F'}\}$ a \textit{Chekanov-type metric} if $d^\mathscr{F}$ and $d^{\mathscr{F}'}$ are both Chekanov-type pseudometrics.
\end{defn*}

For the rest of the paper, we thus set $d=\hat{d}^{\mathscr{F},\mathscr{F}'}$, a metric of Chekanov type associated with families $(\mathscr{F},\mathscr{F}')$ and which is dominated by the Lagrangian Hofer metric $d_H$, i.e. there exists $C>0$ such that $d\leq Cd_H$. \par

\subsection{Useful results from previous work} \label{subsec:previous_results}
We now recall some results from previous work for ease of reference later on. We will not write said results in full generality, but only in the setting which is required for this paper. \par

Below, we make use of the Hausdorff metric $\delta_H$, defined between two compact subsets $A$ and $B$ of $M$ as
\begin{align*}
	\delta_H(A,B):=\inf\left\{\epsilon>0\ \middle|\ A\subseteq B_\epsilon(B),\ B\subseteq B_\epsilon(A)\right\},
\end{align*}
where $B_\epsilon(A):=\cup_{x\in A}B_\epsilon(x)$, i.e.\ it is the $\epsilon$-neighbourhood of $A$. \par

\begin{thm}[\cite{Chasse2023}] \label{thm:first}
	There exist constants $C_1,R_1>0$ with the following property. For all $L$ and $L'$ in $\mathscr{L}^\star_k$ such that $d(L,L')< R_1$, we have that
	\begin{align*}
	\delta_H(L,L')\leq C_1\sqrt{d(L,L')},
	\end{align*}
	where $\delta_H$ denotes the classical Hausdorff distance.
\end{thm}

\begin{thm}[\cite{Chasse2022}] \label{thm:second}
	For all $L$ in $\mathscr{L}_k$, there exist constants $C_2,R_2>0$ with the following property. Whenever $L'\in\mathscr{L}_k$ is such that $\delta_H(L,L')<R_2$, there exists a $C^{1,\alpha'}$-small closed 1-form $\sigma$ of $L$ such that $L'=\operatorname{graph}\sigma$ in a Weinstein neighbourhood of $L$. Furthermore, if $\sigma$ is exact, then
	\begin{align*}
		d(L,L')\leq C_2\delta_H(L,L').
	\end{align*}
	Moreover, if a sequence $\{L_i\}\subseteq\mathscr{L}_k$ has Hausdorff limit $N$, then $N$ is an embedded $C^{1,\alpha}$-Lagrangian submanifold, and there exist diffeomorphisms $f_i:N\xrightarrow{\sim} L_i$ for $i$ large such that $f_i\to\Id$ in the $C^{1,\alpha'}$-topology, $0<\alpha'<\alpha<1$.
\end{thm}

\begin{rem} \label{rem:second}
	In~\cite{Chasse2022}, the limit $N$ was possibly immersed, even in the exact case. However, this was because we were working with the weaker condition of having bounded volume, instead of being $\epsilon$-tame. The latter condition ensures that Hausdorff limits are indeed embedded. \par
	
	In the nonexact case, this is also why the additional condition that $L'$ and $L$ have the same first Betti number is no longer required here: if a sequence $\{L_i\}$ with Betti number $b_1(L_i)=b$ were to Hausdorff-converge to a Lagrangian submanifold $L$ with $b_1(L)<b$, then we would have a sequence of embeddings $\{N\hookrightarrow L_i\}$ which $C^{1,\alpha'}$-converges to a nontrivial cover $N\to L$, which is not possible when each $L_i$ is $\epsilon$-tame for the same $\epsilon>0$.
\end{rem}

\subsection{Applying Theorem~\ref{thm:second}} \label{subsec:apply_thm2}
To study $\mathscr{L}^\star_k$, it will be important to know that the 1-form $\sigma$ appearing in Theorem~\ref{thm:second} is exact. When $L$ is exact and its Weinstein neighbourhood $\Psi:D^*_r L\to M$ is exact, i.e.\ $\Psi^*\lambda=\lambda_0+dF$ for some $F:T^*L\to \R$, this is self-evident. We now extend this to monotone Lagrangian submanifolds bounding enough disks. \par

\begin{lem} \label{lem:apply_thm2}
	Let $L,L'\in\mathscr{L}^{m(\rho)}_k$. Suppose that $L'=\operatorname{graph}\sigma$ in a Weinstein neighbourhood $\Psi$ of $L$. Then, $\sigma$ is exact.
\end{lem}
\begin{pr*}
	Suppose that $\sigma$ is not exact. Then, there exists a loop $\gamma:S^1\to L$ such that $\langle\sigma,\gamma\rangle\neq 0$. In particular, it must be so that $\gamma$ represents a nonzero class in $H_1(L;\Z)^\mathrm{free}$. Since $L$ bounds enough disks, there is thus a disk $D$ in $M$ whose boundary lies on $L$ and is some iterate $\gamma^m$ of $\gamma$. \par
	
	Let $u:S^1\times [0,1]\to T^*L$ be the cylinder given by $(t,s)\mapsto (\gamma(mt),s\sigma(\gamma(mt)))$. Note that $u$ has nonzero area:
	\begin{align*}
		\int_{S^1\times [0,1]}u^*\omega_0=\int_{S^1}(\sigma\circ\gamma^m)^*\lambda_0=\int_{S^1}(\gamma^m)^*\sigma=m\int_{S^1}\gamma^*\sigma\neq 0,
	\end{align*}
	where we have made use of the fact that the zero-section is $\lambda_0$-exact and of the tautological property of $\lambda_0$, i.e.\  $\sigma^*\lambda_0=\sigma$. Since $\Psi$ preserves the symplectic form, $C=\Psi(u(S^1\times [0,1]))$ has nonzero area in $M$. \par
	
	Then, we have that
	\begin{align*}
		\omega(D)=\rho \mu(D),
	\end{align*}
	but also
	\begin{align*}
		\omega(D)+\omega(C)=\omega(D\# C)=\rho\mu(D\# C)=\rho \mu(D),
	\end{align*}
	which is obviously a contradiction. Therefore, $\sigma$ must be exact.
\end{pr*}

We now show that Theorem~\ref{thm:second} also applies to the space of graphs $\mathscr{L}^\Gamma_k=\mathscr{L}^\Gamma_k(M\times M)$, even when that space is equipped with the metric
\begin{align*}
	d'_H(\graph\phi_1,\graph\phi_2):=||\phi_1\phi_2^{-1}||_H,
\end{align*}
where $||\cdot||_H$ is the Hofer norm on $M$. This is not clear at all that this is the case since Theorem~\ref{thm:second} requires that $d\leq d_H$, but we have here that $d=d'_H\geq d_H$. \par

\begin{prop} \label{prop:apply_thm2_graphs}
	For all $L$ in $\mathscr{L}^\Gamma_k$, there exist constants $C_2,R_2>0$ with the following property. Whenever $L'\in\mathscr{L}^\Gamma_k$ is such that $\delta_H(L,L')<R_2$, there exists a $C^{2,\alpha'}$-small function $f$ of $L$ such that $L'=\operatorname{graph} df$ in a Weinstein neighbourhood of $L$. \par
	
	Furthermore, if a sequence of graphs $\{L_i\}\subseteq\mathscr{L}^\Gamma_k$ Hausdorff-converges to another graph $L$, then $L_i\to L$ in $d'_H$.
\end{prop}
\begin{pr*}
	Since $L=(\Id\times\phi)(\Delta)$ for some $\phi\in\Ham(M)$ and $(\Id\times\phi)(\mathscr{L}^\Gamma_k)\subseteq\mathscr{L}^\Gamma_{k'}$ for some $k'$, it suffices to prove the statement for $L=\Delta$ and $R_2$ such that $B_{R_2}(\Delta)$ is contained in a small enough Weinstein neighbourhood of $\Delta$. The proof follows from the typical flux argument, but we give here the details. \par
	
	Fix a Weinstein neighbourhood $\Psi:D_R^*M\to M\times M$ of the diagonal, and let $\{L_i=\graph\phi_i\}\subseteq\mathscr{L}^\Gamma_k$ be such that $L_i\subseteq \Psi(D_{r_i}^*L)$, where $\{r_i\}\subseteq (0,R]$ is a decreasing sequence converging to 0. Then, $L_i\to\Delta$ in the Hausdorff metric (see Lemma~1 of~\cite{Chasse2022}), and there are 1-forms $\sigma_i$ on $M$ such that $L_i=\Psi(\graph\sigma_i)$ for $i$ large by Theorem~\ref{thm:second}. \par
	
	In fact, the theorem implies that $\{\sigma_i\}$ must $C^1$-converge to 0. Therefore, for $i$ large, $\Psi(\graph t\sigma_i)$ is a graph for all $t\in [0,1]$, and we can define $\{\psi^t_i\}_{t\in [0,1]}\subseteq \Symp(M)$ via $\graph\psi^t_i=\Psi(\graph t\sigma_i)$. A direct computation gives $\mathrm{Flux}(\{\psi^t_i\})=[\sigma_i]$. But fix a Hamiltonian isotopy $\{\phi^t_i\}$ with $\phi^1_i=\phi_i$, then the concatenation $\Psi_i=\psi_i\#\overline{\phi_i}$ of $\{\psi^t_i\}$ with $\{\phi^{1-t}_i\}$ is a loop in $\Symp(M)$. Therefore,
	\begin{align*}
		\mathrm{Flux}(\Psi_i)=\mathrm{Flux}(\{\psi^t_i\})-\mathrm{Flux}(\{\phi^t_i\})=[\sigma_i]
	\end{align*}
	is in the flux group $\Gamma_\omega$ of $M$. By the flux conjecture (proved by Ono~\cite{Ono2006}), $\Gamma_\omega$ is discrete. Since $\sigma_i\to 0$, we must thus have that $[\sigma_i]=0$, i.e.\ $\sigma_i=df_i$, for $i$ large. \par
	
	To get the neighbourhood, just note that if such a neighbourhood did not exist, we could construct a sequence $\{L_i\}$ converging to the diagonal, but such that every $\sigma_i$ is nonexact, which would be a contradiction. \par
	
	The last statement follows directly from what we have already said: $L_i=\graph\phi_i\to L=\graph\phi$ in the Hausdorff metric, then it must be that $\phi_i\to\phi$ in the $C^1$-topology, and thus also in the Hofer norm.
\end{pr*}

\section{Topological and metric properties of the geometrically bounded spaces} \label{sec:topo}
We now prove some topological and metric properties of $\mathscr{L}^\star_k$ and their various metric completions. More precisely, in Subsection~\ref{subsec:cpt}, we prove Theorem~\ref{thmout:compact-comp} and some of its direct consequences. In Subsection~\ref{subsec:connect}, we then prove Theorem~\ref{thmout:local_cont}. We also use some comparison results to get a weaker version of that theorem which holds in any Riemannian metric. We conclude in Subsection~\ref{subsec:geodesics} with an analysis of local Hausdorff geodesics. We also use this to conclude that many natural variations of the Hausdorff on $\mathscr{L}^\star_k$ are in fact equivalent on those spaces. \par

\subsection{Compactness and metric completions} \label{subsec:cpt}
We now move to prove results about the compactness of the various metric completions of $\mathscr{L}^\star_k$. \par

It follows directly from Theorems~\ref{thm:first} and~\ref{thm:second} that every $L\in \mathscr{L}^\star_k$ has a neighbourhood where both metrics are equivalent. We thus get directly the following result. \par

\begin{cor} \label{cor:homeo_pre}
	The topology on $\mathscr{L}^\star_k$ induced by $\delta_H$ is equivalent to the one induced by $d$.
\end{cor}

The new observation is that we can actually extend this equivalence to the completions. \par

\begin{prop} \label{prop:homeo_post}
	The metric completions of $\mathscr{L}^\star_k$ in $\delta_H$ and $d$ are homeomorphic.
\end{prop}
\begin{pr*}
	Recall that the metric completion $\widehat{\mathscr{L}}^\star_k$ of $\mathscr{L}^\star_k$ in $d$ is defined as the space of Cauchy sequences in $\mathscr{L}^\star_k$ up to equivalence. Two Cauchy sequences $\{L_i\}$ and $\{L'_j\}$ are called equivalent if for all $\epsilon>0$ there exist $I\in\N$ such that $d(L_i,L'_j)<\epsilon$ for all $i,j\geq I$. For our result, it thus suffices to show that $d$ and $\delta_H$ have the same Cauchy sequences and the same notion of equivalence between them. \par
	
	By Theorem~\ref{thm:first}, $d$-Cauchy sequences are also $\delta_H$-Cauchy sequences. Likewise, when two $d$-Cauchy sequences are $d$-equivalent, they are also $\delta_H$-equivalent. \par
	
	Suppose that $\{L_i\}$ is a $\delta_H$-Cauchy sequence. Denote by $N$ its Hausdorff limit, and fix $\epsilon>0$. By Theorem~\ref{thm:second}, we know that $N$ is actually an embedded Lagrangian $C^{1,\alpha}$-submanifold. Therefore, we can take a sequence $\{L'_i\}$ such that
	\begin{enumerate}[label=(\roman*)]
		\item $d_H(L_i,L'_i)\leq \epsilon$;
		\item $\{L'_i\}$ Hausdorff-converges to a smooth submanifold $N'$;
		\item $\delta_H(N,N')\leq\epsilon$.
	\end{enumerate}
	For example, this can be done by taking a sequence of generic Hamiltonians $\{H_i\}$ with $||H_i||\leq\epsilon$ which is also $C^1$-Cauchy and then taking $L'_i=\phi^{H_i}(L_i)$. If $\star=\Gamma$, we also can suppose that $N'$ is a graph. \par
	
	But since $N'$ is smooth, then $L'_i\xrightarrow{\delta_H}N'$ implies that $L'_i\xrightarrow{d}N'$ by Theorem~\ref{thm:second}~---~or by Proposition~\ref{prop:apply_thm2_graphs} when working with graphs and $d=d'_H$. Therefore, $\{L'_i\}$ is also $d$-Cauchy, and we get that
	\begin{align*}
		d(L_i,L_j)\leq d_H(L_i,L'_i)+d(L'_i,L'_j)+d_H(L'_j,L_j)\leq 3\epsilon
	\end{align*}
	for $i$ and $j$ large. The sequence $\{L_i\}$ is thus itself $d$-Cauchy. The proof that both metrics have the same notion of equivalence between Cauchy sequences is analogous.
\end{pr*}

Noting that the space of closed subsets of $W_k$ is compact in the Hausdorff metric and that the completion of $\mathscr{L}^\star_k$ in the Hausdorff metric naturally embeds in this space as a closed subspace, we get directly the following corollary. \par

\begin{cor} \label{cor:cpt}
	The metric completion $\widehat{\mathscr{L}}^\star_k$ of $\mathscr{L}^\star_k$ in $d$ is compact.
\end{cor}

This thus completes the proof of Theorem~\ref{thmout:compact-comp}. Note that the compactness result also implies that the uncompleted space $\mathscr{L}^\star_k$ is precompact in $\widehat{\mathscr{L}}^\star_k$, so that we get from the generalized Heine-Borel theorem the following. \par

\begin{cor} \label{cor:totally_bnd}
	The space $\mathscr{L}^\star_k$ is totally bounded in $d$, i.e.\ for every $\epsilon>0$, every cover of $\mathscr{L}^\star_k$ by $d$-balls of radius $\epsilon$ admits a finite subcover.
\end{cor}

As we shall see below, this is the statement that will be useful for the corollaries appearing in the introduction. \par

\subsection{Local contractibility} \label{subsec:connect}
We finally show some local path-connectedness properties for $\mathscr{L}^\star_k$ and $\widehat{\mathscr{L}}^\star_k$ with some additional hypotheses on $M$. In fact, in each case, we end up proving something stronger: the spaces are locally contractible. \par

Note that this subsection makes heavy use of Riemannian-geometric results, which we have decided to keep for a dedicated appendix at the end of this paper. This thus makes this part much more readable if one is willing to accept those technical results. \par

\subsubsection*{The Sasaki metric case}
We first show that given $L$, the metric may be chosen so that we have a system of contractible neighbourhoods, i.e.\ we prove the first part of Theorem~\ref{thmout:local_cont}. \par

\begin{prop} \label{prop:path_connect_sasaki}
	Suppose that $g=g_J$ is such that, in a Weinstein neighbourhood of $L\in\mathscr{L}^\star_k$, it corresponds to the Sasaki metric of $g|_L$. Then, $L$ possesses a system of contractible neighbourhoods in $\mathscr{L}^\star_k$.
\end{prop}
\begin{pr*}
	Given Theorem~\ref{thm:second}, we can identify a sufficiently small neighbourhood of $L$ to a set of $C^{1,\alpha'}$-small exact forms. Furthermore, as the neighbourhood gets smaller, the $C^{1,\alpha'}$-norm of the forms tends to 0~---~this is a consequence of the second part of Theorem~\ref{thm:second}. \par
	
	Therefore, it suffices to prove that such a set of forms must be star-shaped about the origin. The result thus follow from Lemmata~\ref{lem:curvature}, \ref{lem:tame_TL}, and~\ref{lem:dM-to-dTL} below. For convenience of computation, these lemmata are stated for vector fields on $L$ instead of forms, but this is equivalent given the musical isomorphisms of $g|_L$.
\end{pr*}

Note that Proposition~\ref{prop:path_connect_sasaki} extends to the metric completion of $\mathscr{L}^\star_k$. \par

\begin{cor} \label{cor:path_connect_sasaki-comp}
	Let $L$ and $g$ be as in Proposition~\ref{prop:path_connect_sasaki}. The Lagrangian $L$ also possesses a system of contractible neighbourhoods in the metric completion $\widehat{\mathscr{L}}^\star_k$ in $d$.
\end{cor}
\begin{pr*}
	By the proof of Proposition~\ref{prop:path_connect_sasaki}, we know that if $L'=\operatorname{graph}dH$ is smooth and in a small enough neighbourhood, then $tL'=\operatorname{graph}tdH$ stays in that neighbourhood for all $t\in [0,1]$. \par
	
	Suppose now that $L'\in\widehat{\mathscr{L}}^\star_k$ is not smooth but in the (completion of the) same neighbourhood. By definition, there is thus a sequence $\{L_i\}\subseteq\mathscr{L}^\star_k$ in that neighbourhood such that $L_i\to L$ in the Hausdorff metric. But then, $tL_i$ stays in it for all $t\in [0,1]$. By continuity of multiplication by a scalar, we have that $tL_i\to tL'$ in the Hausdorff metric. Therefore, $tL'$ stays in the same neighbourhood as $L'$, which gives the result.
\end{pr*}

\subsubsection*{The two-dimensional case}
Even though we expect any $L$ in $\mathscr{L}^\star_k$ to have a system of contractible neighbourhoods for any metric $g$, the computations involved quickly become too complex to handle. An exception to this is when $\dim M=2$. Namely, we can prove the following, which corresponds to the second part of Theorem~\ref{thmout:local_cont}. \par

\begin{prop} \label{prop:contractible-nbhd_2d}
	Let $\dim M=2$ and $k\in\N$. Every $L\in \mathscr{L}^\star_k$ admits a system of contractible neighbourhoods.
\end{prop}

To do so, we employ a similar approach to the Sasaki case. We can do this because, in dimension 2, a tubular neighbourhood of $L$ admits some fairly nice coordinates given by
\begin{center}
	\begin{tikzcd}[row sep=0pt,column sep=1pc]
		\phi\colon (0,\ell)\times (-r,r) \arrow{r} & M \\
		{\hphantom{\phi\colon{}}} (s,t) \arrow[mapsto]{r} & \exp_{\gamma(s)}(tJ\dot{\gamma}(s))
	\end{tikzcd},
\end{center}
where $\gamma:[0,\ell)\to L$ is a parametrization such that $|\dot{\gamma}|\equiv 1$. Note that
\begin{align*}
	\phi^*g &=|W|^2ds^2+dt^2, \\
	\phi^*J &= |W|\frac{\del}{\del t}\otimes ds-\frac{1}{|W|}\frac{\del}{\del s}\otimes dt, \\
	 \text{and}\qquad \phi^*\omega &= |W|ds\wedge dt=d\left(\left(-\int_0^t|W|d\tau\right)ds\right),
\end{align*}
where $W(s,t)$ is the value at time $t$ of the unique Jacobi field along the geodesic $t\mapsto\phi(s,t)$ such that $W(s,0)=\dot{\gamma}(s)$ and $\dot{W}(s,0)=-\kappa(s)\dot{\gamma}(s)$. Here, $\kappa$ is the (signed) geodesic curvature of $L$, which is defined via the relation $\ddot{\gamma}:=\nabla_{\dot{\gamma}}\dot{\gamma}=\kappa J\dot{\gamma}$. \par

In these coordinates, a graph $L'=\{(s,\xi(s))\ |\ s\in [0,\ell)\}=:\operatorname{graph}\xi$ is in the same Hamiltonian isotopy class in $S^1\times (-r,r)$ as $L=\operatorname{graph} 0$ if and only if $s\mapsto \int_0^{\xi(s)}|W|dt$ admits a primitive, which in turn is equivalent to
\begin{align} \label{eqn:condn-exact_2d}
	\int_0^\ell\int_0^{\xi(s)}|W|dtds=0.
\end{align}
In particular, even if $\operatorname{graph}\xi$ is exact, one should not expect $\operatorname{graph}\alpha\xi$ to be for $\alpha\in (0,1)$. We can however circumvent this problem by slightly adjusting our approach. \par

\begin{lem} \label{lem:good-c_2d}
	Suppose that $||\xi||:=\max|\xi|<\frac{r}{2}$. Then, for any $\alpha\in [0,1]$, there exists a unique real number $c(\alpha)$ in $(-\frac{r}{2},\frac{r}{2})$ such that $\xi_\alpha:=\alpha\xi+c(\alpha)$ defines an exact graph. Furthermore, $c$ depends continuously on $\alpha$ and $|c(\alpha)|\leq \alpha ||\xi||$. In particular, the path $\alpha\mapsto \operatorname{graph}\xi_\alpha$ is Hausdorff-continuous and stays in the image of $\phi$. \par
	
	Moreover, if $||\xi||<\frac{r}{3}$, then we have that
	\begin{align*}
		|c(\alpha)-c(\alpha')|\leq ||\xi||\ |\alpha-\alpha'|.
	\end{align*}
\end{lem}
\begin{pr*}
	Fix $\alpha\in [0,1]$. For each $s\in [0,\ell)$, the map $\tau\mapsto \int_0^{\alpha\xi(s)+\tau}|W|dt$ is increasing on $(-\frac{r}{2},\frac{r}{2})$. Indeed, $|W|>0$ for all $(s,t)\in [0,\ell)\times (-r,r)$ since $\phi$ is a chart, and $|\alpha\xi(s)+\tau|\leq ||\xi||+\frac{r}{2}<r$. Furthermore, the map is positive if $\tau>\alpha||\xi||$ and negative if $\tau<-\alpha||\xi||$ for the same reason. Therefore, the same holds for the function $\tau\mapsto \int_0^\ell\int_0^{\alpha\xi(s)+\tau}|W|dtds$. In particular, there is a unique solution $\tau=c(\alpha)$ in $(-\frac{r}{2},\frac{r}{2})$ to (\ref{eqn:condn-exact_2d}), that is such that
	\begin{align*}
		\int_0^\ell\int_0^{\alpha\xi(s)+c(\alpha)}|W|dtds=0.
	\end{align*}

	The continuity of $c$ follows directly from the fact that $\int_0^\ell\int_0^{\alpha\xi(s)+\tau}|W|dtds$ depends continuously on $\alpha$. The estimate on the value of $c$ follows from the fact that this integral is positive if $\tau>\alpha||\xi||$ and negative if $\tau<-\alpha||\xi||$.
	
	The final estimate follows from applying the same logic as above to the function
	\begin{align*} 
		\tau\mapsto \int_0^\ell\int_0^{(\alpha-\alpha')\xi(s)+\tau-c(\alpha')}|W|dt
	\end{align*}
	and noting that $\tau=c(\alpha)$ must be its unique zero in $(-\frac{r}{3},\frac{r}{3})$.
\end{pr*}

Therefore, we are precisely in the setting of Subsection~\ref{subsec:geo_2d}, and Proposition~\ref{prop:contractible-nbhd_2d} follows directly from applying Lemmata~\ref{lem:curvature_2d} and~\ref{lem:tameness_2d} to the path $\alpha\mapsto \xi_\alpha$.

\subsubsection*{The general case}
We now partially extend the local path-connectedness result to other metrics than the Sasaki ones. The results are of course weaker in this context, but they still point in the same direction as the locally Sasaki case. More precisely, we get the following. \par

\begin{prop} \label{prop:contractible_nbhd_general}
	For every $k>0$ and $L\in\mathscr{L}^\star_k$, there are $a\geq 1$ and $b\geq 0$ with the following property. The Lagrangian submanifold $L$ possesses a (system of) neighbourhood $U$ in $\mathscr{L}^\star_{k}$ such that the inclusion $U\hookrightarrow \mathscr{L}^\star_{ak+b}$ is nullhomotopic.
\end{prop}

This follows directly from combining the Sasaki case (Proposition~\ref{prop:path_connect_sasaki}) with the comparison results Lemmata~\ref{lem:curvature_comparision} and~\ref{lem:tameness_comparision} below.  \par

\subsection{Local geodesics} \label{subsec:geodesics}
We now turn our attention to the geodesics in the Hausdorff metric in a neighbourhood of a Lagrangian submanifold $L\in\mathscr{L}^\star_k$. \par

We first recall the definition of a geodesic in a general metric space.

\begin{defn*}
	Let $(X,d)$ be a metric space. The \emph{length} of curve $c:[a,b]\to X$ is given by
	\begin{align*}
		\ell(c):=\sup_{a=t_0<\dots<t_\ell=b}\sum_{i=1}^\ell d\left(c(t_{i-1}),c(t_i)\right)\in [0,\infty].
	\end{align*}
	We say that $c$ has \emph{constant speed} if there exist $\lambda\geq 0$ such that $\ell(c|_{[t,s]})=\lambda|t-s|$ for all $a\leq t\leq s\leq b$. In that case, we call $\lambda$ its \emph{speed}. \par
	
	A \emph{geodesic} is a curve $c:[a,b]\to X$ which has constant speed and is locally minimizing in $d$, i.e.\ for every $t_0\in [a,b]$, there is some $\epsilon>0$ such that whenever $t\leq s\in [a,b]\cap(t_0-\epsilon,t_0+\epsilon)$, then
	\begin{align*}
		\ell\left(c|_{[t,s]}\right)=d(c(t),c(s)).
	\end{align*}
	If the above equality holds for all $t,s\in [a,b]$, then we call $c$ a \emph{minimizing geodesic}.
\end{defn*}

We begin with a description of certain geodesics in the Hausdorff distance in a small enough neighbourhood of any submanifold.
\begin{prop} \label{prop:radial_distance}
	Let $N$ be a submanifold of a complete Riemannian manifold $M$ with tubular neighbourhood $U$, i.e.\ there is a neighbourhood $V$ of the zero-section of $TN^\perp$ such that the exponential gives a diffeomorphism $V\xrightarrow{\sim} U$. Suppose that $V=B_\epsilon(N)$ for some $\epsilon>0$, that is, $V$ is the $\epsilon$-neighbourhood of the zero-section in $TN^\perp$. If $N'\subseteq U$ is a submanifold such that $N'=\exp\sigma(N)$ for some section $\sigma$ of $TN^\perp$, then
	\begin{align*}
		\delta_H(tN',sN')=|t-s|\max|\sigma|
	\end{align*}
	for all $t,s\in [0,1]$, where $tN':=\exp t\sigma(N)$.
\end{prop}

We get directly from this a characterization of the radial Hausdorff-geodesics.
\begin{cor} \label{cor:radial_geodesics}
	If $N$ and $N'$ are as above, the path $t\mapsto tN'$ is a minimal geodesic in the Hausdorff metric. In particular, if the Riemannian metric corresponds to the Sasaki metric on a Weinstein neighbourhood of $L\in\mathscr{L}^\star_k$, then $L$ possesses a system of geodesically-starshaped neighbourhoods in $(\mathscr{L}^\star_k,\delta_H)$.
\end{cor}
\begin{pr*}
	Denote by $c$ the path, and take $0\leq a<b\leq 1$. Then, the length of $c|_{[a,b]}$ is given by
	\begin{align*}
		\ell\left(c|_{[a,b]}\right)
		&=\sup_{a=t_0<\dots<t_\ell=b}\sum_{i=1}^\ell \delta_H(c(t_{i-1}),c(t_i)) \\
		&=\sup_{a=t_0<\dots<t_\ell=b}\sum_{i=1}^\ell (t_i-t_{i-1})\max|\sigma| \\
		&=(b-a)\max|\sigma| \\
		&=\delta_H(c(a),c(b)),
	\end{align*}
which proves the first part of the result. \par

The statement on geodesically-starshaped then follows directly, knowing that $tL'$ stays in $\mathscr{L}^\star_k$ by Proposition~\ref{prop:path_connect_sasaki}.
\end{pr*}

\begin{proof}[Proof of Proposition~\ref{prop:radial_distance}]
	We first note that for every $t,s\in [0,1]$
	\begin{align} \label{eqn:s-tN-sN}
		s(tN';sN') &=\max_{x\in N}d_M(\exp t\sigma(x),sN') \\ \nonumber
		&\leq \max_{x\in N}d_M(\exp t\sigma(x),\exp s\sigma(x))= |t-s|\max |\sigma|,
	\end{align}
	since the exponential on $V$ is a radial isometry. In particular, we have that $\delta_H(tN',sN')\leq|t-s|\max|\sigma|$. \par
	
	Suppose that $s<t$, and let $x_0\in N$ be such that $|\sigma(x_0)|=\max|\sigma|$. Suppose that there exists $y\in N$ such that $d_M(\exp t\sigma(x_0),\exp s\sigma(y))<d_M(\exp t\sigma(x_0),\exp s\sigma(x_0))=(t-s)|\sigma(x_0)|$. Then, we have that
	\begin{align*}
		t|\sigma(x_0)|
		&= (t-s)|\sigma(x_0)|+s|\sigma(x_0)| \\
		&\geq d_M(\exp t\sigma(x_0),\exp s\sigma(x_0))+s|\sigma(y)| \\
		&> d_M(\exp t\sigma(x_0),\exp s\sigma(y))+d_M(y,\exp s\sigma(y)) \\
		&\geq d_M(\exp t\sigma(x_0),y)
	\end{align*}
	This means that
	\begin{align*}
		d_M(\exp t\sigma(x_0),N)\leq d_M(\exp t\sigma(x_0),y)< t|\sigma(x_0)|.
	\end{align*}
	Let $\gamma:[0,1]\to M$ be a minimal geodesic from $N$ to $\exp t\sigma(x_0)$. Since $N$ is closed, $\gamma'(0)\in TN^\perp$, so that $\gamma(t)=\exp(t\gamma'(0))$. But then, $|\gamma'(0)|=d_M(\exp t\sigma(x_0),N)<t|\sigma(x_0)|<\epsilon$. Therefore, $\gamma'(0)$ and $t\sigma(x_0)$ are two vectors in $V$ whose image under the exponential map is $\exp t\sigma(x_0)$, which is a contradiction with the hypothesis that $\exp|_V$ be a diffeomorphism onto its image. The inequality (\ref{eqn:s-tN-sN}) is thus in fact an equality, which proves the lemma.
\end{proof}

In the two-dimensional case, things are not as straightforward. Indeed, it is easy to see that the above proof gives that
\begin{align*}
	\delta_H(\graph\xi_\alpha,\graph\xi_{\alpha'})=\max |\xi_\alpha-\xi_{\alpha'}|,
\end{align*}
which means that we should not expect $\alpha\mapsto \xi_\alpha$ to be a Hausdorff-geodesic in general. However, the above equality together with the estimate on the Lipschitz constant of $c(\alpha)$ in Lemma~\ref{lem:good-c_2d} gives the following.

\begin{lem} \label{lem:radial-distance_2D}
	Suppose that $\dim M=2$. Every $L\in\mathscr{L}^\star_k$ has a neighbourhood $U$ in $\mathscr{L}^\star_k$ such that
	\begin{align*}
		\delta_H(\graph\xi_\alpha,\graph\xi_{\alpha'})\leq 2|\alpha-\alpha'|\max |\xi|=2|\alpha-\alpha'|\delta_H(L,\graph\xi)
	\end{align*}
	whenever $\graph\xi\in U$.
\end{lem}

\subsubsection*{Variations on the Hausdorff metric}
Following a result of Sosov~\cite{Sosov2001}, the Hausdorff metric between $L$ and $L'$ is given by the infinimum over all $\delta_H$-continuous paths of closed subsets from $L$ to $L'$. In fact, this infimum is even realized by a geodesic. This is because we have chosen the Riemannian metric on $M$ so that $(M,d_M)$ is a complete, geodesic metric space. Therefore, his definition of a geodesic corresponds to ours. \par

In this context, it is thus natural to consider what happens when we take the infimum over paths in a smaller set. More precisely, we are interested in the two following variants of the usual Hausdorff metric on $\mathscr{L}^\star_k$:
\begin{align*}
	\delta_H^\mathrm{Man}(L_0,L_1)
	&:=\inf\left\{\ell(c)\ |\ c(i)=L_i,\ c(t)\text{ is a n-dimensional manifold }\forall t\in [0,1]\right\}, \\
	\delta_H^{(\star,k)}(L_0,L_1)
	&:=\inf\left\{\ell(c)\ |\ c(i)=L_i,\ c(t)\in\mathscr{L}^\star_k\ \forall t\in [0,1]\right\}.
\end{align*}
Here, all $c$'s are $\delta_H$-continuous, and all manifolds are smooth, closed, and connected. Note that
\begin{align*}
	\delta_H\leq \delta_H^\mathrm{Man}\leq \delta_H^{(\star,k)}.
\end{align*}
The first part of this subsection shows that, at least locally, these inequalities are equalities in good cases. \par

\begin{prop} \label{prop:var-Hausdorff_local}
	Every $L\in\mathscr{L}^\star_k$ has a neighbourhood $U$ in $\mathscr{L}^\star_k$ such that for all $L'\in U$, the following holds.
	\begin{enumerate}[label=(\roman*)]
		\item $\delta_H^\mathrm{Man}(L,L')=\delta_H(L,L')$.
		\item If the Riemannian metric of $M$ corresponds to the Sasaki metric on a Weinstein neighbourhood of $L$, then $\delta_H^{(\star,k)}(L,L')=\delta_H(L,L')$.
		\item If $\dim M=2$, then $\delta_H^{(\star,k)}(L,L')\leq 2\delta_H(L,L')$.
	\end{enumerate}
\end{prop}
\begin{pr*}
	We take $U$ to be a tubular neighbourhood of $L$. By making $U$ smaller if necessary, we may suppose that all $L'\in\mathscr{L}^\star_k$ such that $L'\subseteq U$ are graphs by Theorem~\ref{thm:second}. Therefore, (i) and (ii) follow directly from Corollary~\ref{cor:radial_geodesics}. Likewise, (iii) follows from Lemma~\ref{lem:radial-distance_2D}. 
\end{pr*}

In particular, we get the following characterization of the topologies induced by the variations of the Hausdorff metric. \par

\begin{cor} \label{cor:var-Hausdorff_topologies}
	The metrics $\delta_H^\mathrm{Man}$ and $\delta_H$ induce the same topology on $\mathscr{L}^\star_k$. If $\dim M=2$, then the same holds for $\delta_H^{(\star,k)}$ and $\delta_H$.
\end{cor}

\section{Symplectic properties of the geometrically bounded spaces} \label{sec:symp}
We now show some symplectic properties of the Lagrangian submanifolds in $\mathscr{L}^\star_k$ which derive from the topological and metric properties proved above. More precisely, we first prove in Subsection~\ref{subsec:ham_classes} Theorem~\ref{thmout:finite_classes}. The rest of the subsections are then dedicated to proving the many corollaries following the first principle appearing in the introduction: one per section and in the same order. \par

\subsection{Hamiltonian isotopy classes} \label{subsec:ham_classes}
We explain how the connected components of $\mathscr{L}^\star_k$ are related to the isotopy classes of the Lagrangian submanifold therein. \par

\begin{prop} \label{prop:isotopy_classes+connected_components}
	For each $k>0$, there exists $A>0$ with the following property. If $L,L'\in\mathscr{L}^\star_k$ are not Hamiltonian isotopic, then
	\begin{align*}
		d(L,L')\geq A.
	\end{align*}
\end{prop}
\begin{pr*}
	Let $\{L_i\}$ and $\{L'_i\}$ be sequences in  $\mathscr{L}^\star_k$ such that $d(L_i,L'_i)$ tends to zero, that is, they are equivalent in $d$. By Proposition~\ref{prop:homeo_post}, they must then be also equivalent in $d_H$. In particular, $d(L_i,L'_i)$ is finite for $i$ large, which implies that they are Hamiltonian isotopic for $i$ large. Thus, such an $A>0$ must exist.
\end{pr*}

Finally, we can similarly get a fairly powerful result on the possible Hamiltonian isotopy classes in $\mathscr{L}^\star_k$. \par

\begin{prop} \label{prop:finitely_isotopy_classes}
	There are finitely many Hamiltonian isotopy classes in $\mathscr{L}^\star_k$.
\end{prop}
\begin{pr*}
	Suppose the contrary. Then, there exists a sequence $\{L_i\}\subseteq \mathscr{L}^\star_k$ with $L_i$ not Hamiltonian isotopic to $L_j$ if $i\neq j$. By Corollary~\ref{cor:cpt}, we may pass to a converging subsequence. But by Proposition~\ref{prop:isotopy_classes+connected_components} above, we will eventually get $d(L_i,L_j)<A$, so that $L_i$ must be Hamiltonian isotopic to $L_j$ for $i$ and $j$ large, and we have a contradiction.
\end{pr*}

We close this section with a simple, but important observation: it is necessary to fix a Liouville form $\lambda$ when $\star=e$ or a monotonicity constant $\rho$ for $\star=m(\rho)$. Likewise, we truly need the ``bounding enough disks'' condition for our results to hold. Indeed, in each case when one of these conditions is broken, we get a counterexample to Proposition~\ref{prop:finitely_isotopy_classes}.
\begin{itemize}
	\item On the flat cylinder $T^*S^1$, each parallel is a totally geodesic 1-tame Lagrangian submanifold which is exact for some primitive of the usual symplectic form. However, that primitive is different for each parallel, i.e.\ only one of them can belong to $\mathscr{L}^e(T^*S^1)$. \par
	\item These parallels can also be seen as monotone Lagrangian submanifold for any $\rho\geq 0$. However, they bound no disk at all, and thus do not respect the condition of bounding enough disks, i.e.\ they never belong to $\mathscr{L}^{m(\rho)}(T^*S^1)$.
	\item In $\R^2$ with its usual structure, the Hamiltonian isotopy class of a circle is determined by the area it encloses. Clearly, for any $k>0$, there is a continuum of possible areas enclosed by a circle in $\mathscr{L}_k(\R^2)$. Furthermore, each of these circles are monotone. However, they are all so for different monotonicity constants, i.e.\ only one class can belong to any $\mathscr{L}^{m(\rho)}(\R^2)$.
\end{itemize}

\begin{rem} \label{rem:isotopy classes+connected_components}
	In dimension 2, Proposition~\ref{prop:finitely_isotopy_classes} follows directly from Proposition~\ref{prop:isotopy_classes+connected_components} and Corollary~\ref{cor:cpt}, since we then know that $\widehat{\mathscr{L}}^\star_k$ is locally path connected by Proposition~\ref{prop:contractible-nbhd_2d}. Indeed, this ensures that the connected components of $\widehat{\mathscr{L}}^\star_k$ are open, and they are thus in finite number, by compactness of the space. But by Proposition~\ref{prop:isotopy_classes+connected_components}, each connected component is contained in a unique Hamiltonian class, so that the latter must also be in finite number.
\end{rem}

\subsection{Boundedness of $d$}
We now prove the corollary on the boundedness of $d$ when it is restricted to one Hamiltonian orbit. \par

\begin{cor} \label{cor:viterbo_conj}
	For every $k\geq 1$, there is some $B>0$ with the following property. Let $L,L'\in\mathscr{L}^\star_k$. Suppose that either $L$ and $L'$ are Hamiltonian isotopic or that $M=T^*N$, $\star=e$, and $d=\gamma$. Then,
	\begin{align*}
		d(L,L')\leq B.
	\end{align*}
\end{cor}
\begin{pr*}
	If $L,L'\in\mathscr{L}^{L_0}_k$ for some $L_0\in\mathscr{L}^{\star}_\infty$, then $d_H(L,L')<\infty$ by definition. Since $d$ is dominated by $d_H$, we thus also have that $d(L,L')<\infty$. Therefore, total boundedness of $\mathscr{L}^{L_0}_k$ (as proved in Corollary~\ref{cor:totally_bnd}) implies boundedness. \par
	
	When $M=T^*N$, $\star=e$, and $d=\gamma$, the same argument works because we then have $d(L,L')<\infty$ for all $L,L'\in\mathscr{L}^\star_k$.
\end{pr*}

\begin{rem} \label{rem:viterbo_conj-improvements}
	The improvement here, compared to the version of the Viterbo conjecture appearing in~\cite{Chasse2022}, is that the bound on $\gamma(L)$ stands for all exact Lagrangian submanifolds in the unit codisk bundle, not just in a codisk bundle of small enough radius. However, the constant $A$ now explicitly depends on $k$. We have however not simply rescaled the previous estimate: the present bound applies to Lagrangian submanifolds which are not graphs, which was not the case previously. 
\end{rem}

\subsection{Graphs and Ostrover's example} \label{subsec:graphs}
In~\cite{Ostrover2003}, Ostrover constructs, for every closed symplectic manifold $M$ such that $\pi_2(M)=0$ and any $c>0$ small enough, a sequence of Hamiltonian diffeomorphisms $\{\phi^c_i\}\subseteq\Ham(M)$ such that
\begin{enumerate}[label=(\roman*)]
	\item $||\phi^c_i||_H\xrightarrow{i\to\infty} \infty$;
	\item $d_H(\Delta,\graph\phi^c_i)\equiv c$,
\end{enumerate}
where $\Delta\subseteq M\times M$ is the diagonal and $d_H$ is the Lagrangian Hofer metric of $M\times M$. In particular, if we set $\phi_i:=\phi^{1/i}_i$, we get a sequence of Hamiltonian diffeomorphisms which Hofer-converges to infinity, but whose graphs Lagrangian--Hofer-converges to the diagonal. In this subsection, we want to show that such a phenomenon is impossible in the world of geometrically bounded Lagrangian graphs. That is, we show the second corollary in the introduction.  \par

Note that contrary to all other subsections of this paper, we \emph{do not} require that our monotone Lagrangian submanifolds bound enough disks. We recall that
\begin{align*}
	\mathscr{L}^\Gamma_k=\left\{L\in \mathscr{L}_k(M\times M)\ \middle|\ L=\graph\phi,\ \phi\in\Ham(M)\right\},
\end{align*}
where $M$ is equipped with some Riemannian metric, and $M\times M$, with the resulting product metric. \par

From Proposition~\ref{prop:apply_thm2_graphs}, Theorem~\ref{thm:second} applies to $\mathscr{L}^\Gamma_k$ equipped with the metric $d'_H$ induced by the Hofer norm, i.e.\ defined by
\begin{align*}
	d'_H(\graph\phi_1,\graph\phi_2):=||\phi_1\phi^{-1}_2||_H,
\end{align*}
where $||\cdot ||_H$ is the Hofer norm of $M$. Since $d'_H\geq d_H$, Theorem~\ref{thm:first} also trivially applies. In particular, we can make use of Proposition~\ref{cor:totally_bnd}. We thus get the following, since $\mathscr{L}^\Gamma_k$ contains a unique Hamiltonian isotopy class.

\begin{cor} \label{cor:ostrover_ex}
	On $\mathscr{L}^\Gamma_k$, $d_H$ and $d'_H$ induce the same topology, Furthermore, $d'_H$ is bounded. In particular, an example {\normalfont à la} Ostrover does not exist in $\mathscr{L}^\Gamma_k$.
\end{cor}

Note that we have a defined notion of $C^1$-distance between graphs and of $C^1$-bounds on them through the diffeomorphisms that define them. We suspect that the spaces resulting from these bounds also obey a result analogous to Corollary~\ref{cor:ostrover_ex} above. \par

However, working with curvature bounds of the graphs allows the limit in the completion to be represented by Lagrangian submanifolds of $M\times M$ which are not graphs. In particular, we get the following. \par

\begin{cor} \label{cor:limit_graphs}
	There are elements in the metric completion of $(\Ham(M),||\cdot||_H)$ which can uniquely be represented by nongraphical ($C^{1,\alpha}$) Lagrangian submanifolds of $M\times M$. 
\end{cor}

It would be quite interesting to be able to detect which elements of the completion have this property. \par

\begin{rem} \label{rem:spectral_graphs}
	One could ask the same question as above but with the Hofer norm replaced by the spectral one. However, this is a trivial question: it is known~\cite{LeclercqZapolsky2018} that in the monotone setting, the spectral norm of a Hamiltonian diffeomorphism in $M$ is equal to the spectral distance of its graph to the diagonal in $M\times M$.
\end{rem}

\subsection{Order of a symplectomorphism and categorical entropy} \label{subsec:entropy}
We now move on to the first corollary of Theorem~\ref{thmout:finite_classes}. That result (or more precisely, Proposition~\ref{prop:finitely_isotopy_classes}) directly implies the following. \par

\begin{cor} \label{cor:order_symp}
	Let $L\in\mathscr{L}^\star_k$, and let $\psi$ be a symplectomorphism of $M$. If there exist $k\geq 1$ such that $\psi^\nu(L)\in\mathscr{L}^\star_k$ for all $\nu\geq 1$, then there exist $N$ such that $\psi^N(L)$ is Hamiltonian isotopic to $L$.
\end{cor}

In fact, we can make the above statement somewhat quantitative through the various notions of entropy. First of all, we note that we have a criterion for the vanishing of barcode entropy~---~we refer the reader to~\cite{CineliGinzburgGurel2022} and~\cite{Dawid2023} for the definitions. \par

\begin{cor} \label{cor:relative-barcode-entropy}
	Let $L$ and $\psi$ be as in Corollary~\ref{cor:order_symp}. If $\psi$ is Hamiltonian and $L'$ is another exact or monotone Lagrangian submanifold, then the relative barcode entropy $\hbar(\psi;L,L')$ vanishes. If $L$, $L'$, and $\psi$ are all exact, then the same holds for the slow relative barcode entropy $\hbar^{\mathscr{sl}}(\psi;L,L')$.
\end{cor}
\begin{pr*}
	By Proposition 4 of~\cite{Chasse2022}, $\psi^\nu(L)\in\mathscr{L}^\star_k$ for all $\nu$ implies a universal bound on the volume of the $\psi^\nu(L)$. The result on the usual barcode entropy then follows directly from the proof of Theorem~2.4 of~\cite{CineliGinzburgGurel2022}. The result on slow barcode entropy follows instead from the proof of Theorem~A of~\cite{Dawid2023}.
\end{pr*}

Perhaps more interestingly however, we gather from Corollary~\ref{cor:order_symp} a geometrical criterion for the vanishing of the so-called categorical entropy of a symplectomorphism, which we define below. We make use of the definition using multiple generators of~\cite{BaeChoaJeongKarabasLee2022}, instead of the original definition using a single split generator~\cite{DimitrovHaidenKatzarkovKontsevich2014}, but it is shown in the former paper that these are equivalent. The reason for this is that we want to work with actual Lagrangian submanifolds, not abstract twisted complexes or modules over the Fukaya category. \par

We first introduce the following notation. Let $\mathscr{C}$ be a (non-graded) triangulated category. For a morphism $f:A\to B$ in $\mathscr{C}$, we denote by $\mathrm{Cone}(f)$ its cone, i.e. the unique-up-to-isomorphism object turning $A\to B\to \mathrm{Cone}(f)\to A$ into a distinguished triangle of $\mathscr{C}$. More generally, we define by induction $\mathrm{Cone}(f_1,\dots,f_m)$ to be the cone of the map $f_m:A_m\to \mathrm{Cone}(f_1,\dots,f_{m-1})$. \par

\begin{defn*}
	Let $\mathscr{C}$ be a non-graded triangulated category, and let $A,G_1,\dots,G_\ell$ be objects of $\mathscr{C}$. The \emph{complexity} of $A$ with respect to $G_1,\dots G_\ell$ is given by
	\begin{align*}
		&\delta(G_1,\dots, G_\ell;A) \\
		&\quad :=\inf\left\{m\ \middle|\ A\oplus A'=\mathrm{Cone}(f_1,\dots,f_m),\ A'\in\mathrm{Ob}(\mathscr{C}),\ \operatorname{dom} f_i\in\{G_1,\dots,G_\ell\}\right\}.
	\end{align*} 
	Furthermore, if $G_1,\dots G_\ell$ split-generate $\mathscr{C}$ and $\Phi$ is an endofunctor of $\mathscr{C}$, we define its \emph{categorical entropy} to be
	\begin{align*}
		h_\mathrm{cat}(\Phi):=\lim_{\nu\to\infty}\frac{\delta(G_1,\dots, G_\ell;\Phi^\nu(G_1\oplus\dots\oplus G_\ell))}{\nu}\in [0,+\infty].
	\end{align*}
\end{defn*}

In other words, complexity measures how many iterated cones are needed to get $A$ from $G_1,\dots, G_\ell$ up to some splitting, whilst categorical entropy measures how much $\Phi$ ``complexifies'' the generators of $\mathscr{C}$. As the notation suggests, the definition of categorical entropy is independent of the choice of split-generators. \par

\begin{rem} \label{rem:split_vs_nonsplit-generators}
	In the definition of complexity above, we could replace the $A\oplus A'=\mathrm{Cone}(f_1,$ $\dots,f_m)$ condition by simply $A=\mathrm{Cone}(f_1,\dots,f_m)$ and then work with generators to define categorical entropy. This is perfectly valid but leads to a number which is~---~in general~---~larger than what we have defined here. Since we are interested in a criterion for the vanishing of entropy, it is a more general approach to work with split-generation.
\end{rem}

In the symplectic context, we take $\mathscr{C}$ to be the derived Fukaya category $\mathrm{DFuk}^\star(M)$ generated by $\mathscr{L}^\star_\infty$~---~this is well defined in both the exact~\cite{Seidel2008} and monotone~\cite{Sheridan2016} settings. Then, any symplectomorphism $\psi$ preserving $\mathscr{L}^\star_\infty$ will induce an endofunctor of $\mathrm{DFuk}^\star(M)$~---~we will call the categorical entropy of that functor the categorical entropy of $\psi$. The following result follows directly from Corollary~\ref{cor:order_symp} since Hamiltonian isotopic Lagrangian submanifolds induce isomorphic objects in the derived Fukaya category. \par

\begin{cor} \label{cor:cat-entropy}
	Suppose that $\psi$ is a symplectomorphism of $M$ preserving $\mathscr{L}^\star_\infty$ such that, for a set of generator $L_1,\dots L_\ell$ of $\mathrm{DFuk}^\star(M)$, $\psi^\nu(L_i)\in\mathscr{L}^\star_k$ for some $k$, for every $\nu$ and every $i$. Then, $h_\mathrm{cat}(\psi)=0$.
\end{cor}

In other words, if $h_\mathrm{cat}(\psi)>0$, then there is some Lagrangian submanifold $L$ which is a factor of a split-generator $G$ of $\mathrm{DFuk}^\star(M)$ such that the sequence $\{\psi^\nu(L)\}$ is not contained in any $\mathscr{L}^\star_k$. Note that we may suppose that such a Lagrangian submanifold $L$ induces a nontrivial object in $\mathrm{DFuk}^\star(M)$. Thus, $\psi$ must then deform some symplectically-important Lagrangian submanifold. \par

\begin{rem} \label{rem:entropy-with-weights}
	There is currently work in progress from Ambrosioni, Biran, and Cornea which defines weighted versions of categorical entropy coming from the triangulated persistence structure of the derived Fukaya category~\cite{BiranCorneaZhang2021}, which associates to cones an associated weight. Because our spaces are all totally bounded, we expect that their notion of entropy is also well behaved in our setting, so that we can expect similar results as above.
\end{rem}

\begin{rem} \label{rem:geo-entropy}
	Lemmata~\ref{lem:curvature_comparision} and~\ref{lem:tameness_comparision} below imply that, for any symplectomorphism $\psi$ preserving $\mathscr{L}^\star_\infty$ and any $L\in\mathscr{L}^\star_\infty$, the quantity
	\begin{align*}
		\eta(\psi;L):=\limsup_{\nu\to\infty}\frac{\log^+\left(\underline{k}(\psi^\nu(L))\right)}{\nu}\in [0,+\infty],
	\end{align*}
	where $\underline{k}(L):=\inf\{k\ |\ L\in\mathscr{L}^\star_k\}$ and $\log^+(x):=\max\{0,\log(x)\}$, is independent on the Riemannian metric $g$ on $M$ or on the choice of compacts $W_k$. However, its relevance is nebulous to us. \par
	
	For example, it is well-known that the quantity 
	\begin{align*}
		\Gamma(\psi;L):=\limsup_{\nu\to\infty}\frac{\log^+\left(\mathrm{Vol}( \psi^\nu(L))\right)}{\nu}\in [0,+\infty]
	\end{align*}
	is a lower bound to topological entropy~\cite{Yomdin1987} and an upper bound to barcode entropy with any $L'\in\mathscr{L}^\star_\infty$~\cite{CineliGinzburgGurel2022}. Furthermore, we have shown (Proposition 4 of~\cite{Chasse2022}) that being in $\mathscr{L}_k$ implies respecting a volume bound. However, that volume bound is generally not polynomial in $k$, so that there is no obvious link between $\eta$ and $\Gamma$~---~and thus between $\eta$ and entropy. \par
	
	Even in the case when $n=1$, in which case the volume bound reduces to
	\begin{align} \label{eq:bound-volume_dim1}
		\mathrm{Vol}(B_{(k+1)^{-1}}(W_k))\geq 2\left\lfloor\mathrm{Diam}(L)\right\rfloor\min\left\{\mathrm{Diam}(L),(k+1)^{-1}\right\},
	\end{align}
	with $\mathrm{Vol}(L)=2\ \mathrm{Diam}(L)$, this only implies that
	\begin{align} \label{eq:geo-to-vol-entropy}
		\eta(\psi;L)\geq \Gamma(\psi;L)-\limsup_{\nu\to\infty}\frac{\log\left(\mathrm{Vol}( W_{\underline{k}(\psi^\nu(L)})\right)}{\nu}.
	\end{align}
	But the right-hand side simply vanishes. Indeed, we have the freedom to choose in (\ref{eq:bound-volume_dim1}) any $W_k$ containing $\psi^\nu(L)$. In particular, let $W_k$ be the ``smallest'' possible choice: the tubular neighbourhood of $\psi^\nu(L)$ of radius $r$ for $r$ small. Then, $\mathrm{Vol}(W_k)$ behaves like $\mathrm{Vol}(\psi^\nu(L))r$~---~see for example Theorem~9.23 of~\cite{Gray2004}~---~and the superior limit is simply $\Gamma(\psi;L)$.
\end{rem}

\subsection{Connected components of $\widehat{\mathscr{L}}^\star_k$}
We now move to the second corollary of Theorem~\ref{thmout:finite_classes}. In fact, we prove the following slightly stronger statement. \par

\begin{cor} \label{cor:nearby_conj}
	If $L$ and $L'$ are smooth Lagrangian submanifolds belonging to the same connected component of $\mathscr{L}^\star_k$ or of $\widehat{\mathscr{L}}^\star_k$, then they are Hamiltonian isotopic. \par
	
	In particular, $L$ and $L'$ are Hamiltonian isotopic if there is a $d$-continuous path in $\widehat{\mathscr{L}}^\star_k$ from $L$ to $L'$.
\end{cor}
\begin{pr*}
	Note that the set $\mathscr{L}^{L_0}_k=(\Ham(M)\cdot L_0)\cdot \mathscr{L}^\star_k$ must be a clopen of $\mathscr{L}^\star_k$. Indeed, the fact that $d(L,L')\geq A>0$ whenever $L'\in\mathscr{L}^\star_k$ is not Hamiltonian isotopic to $L$ implies that $\mathscr{L}^{L_0}_k$ contains all of its limit points in $\mathscr{L}^\star_k$, i.e.\ $\mathscr{L}^{L_0}_k$ is closed. But that fact also implies that $\mathscr{L}^{L_0}_k$ is equal to the union of the metric balls of radius $\frac{A}{2}$ centered at points on $\mathscr{L}^L_k$, so that it must also be open. The conclusion then follows from the fact that a clopen always fully contains the connected components of its points. \par
	
	For the last statement, simply note that the path-connected component of a point is always contained in its connected component.
\end{pr*}

\begin{rem} \label{rem:exact-iso_C0}
	Obviously, every exact Lagrangian isotopy $\{L_t\}_{t\in [0,1]}$ respects the hypotheses of Corollary~\ref{cor:nearby_conj}. However, these hypotheses are strictly more general. For example, it is proven in~\cite{Jannaud2021} that if $H^1(N;\R)=0$, then the Floer barcode is $C^0$-continuous. That is, if $t\mapsto \phi_t$ is a $C^0$-continuous path in the $C^0$-completion of the group of symplectomorphisms of $T^*N$ and $L,L'\subseteq T^*N$ are exact, then the Floer barcode $\mathcal{B}(\phi_t(L),L')$ depends continuously of $t$ in the bottleneck distance. In particular, this means that such a $C^0$-continuous path $t\mapsto \phi_t(L)$ through Lagrangian submanifolds of $\mathscr{L}^e_k$ respects the hypotheses of Corollary~\ref{cor:nearby_conj}.
\end{rem}

\subsection{Hofer geodesics} \label{subsec:hofer-geo}
We now finally move on to the corollary of Theorem~\ref{thmout:local_cont} in the introduction. \par

We recall that it has been proven by Milinkovi\'c~\cite{Milinkovic2001} that the Hofer and spectral distances between two graphs in $T^*L$ are both given by
\begin{align} \label{eqn:hofer-distance_graphs}
	d_H(\operatorname{graph}df,\operatorname{graph}dg)=\gamma(\operatorname{graph}df,\operatorname{graph}dg)=\max |f-g|-\min|f-g|.
\end{align}
In particular, if $L'$ is a graph, then the path $t\mapsto tL'$ is a minimizing geodesic from $L$ to $L'$ in the Hofer metric of $T^*L$. \par

However, when $T^*L$ is embedded in a symplectic manifold $M$ via a Weinstein neighbourhood $\Psi$ of some Lagrangian submanifold $L$, (\ref{eqn:hofer-distance_graphs}) is reduced to a simple bound in the Hofer metric of $M$. In particular, we are no longer guaranteed that $t\mapsto \Psi(tL')$ is minimizing when $L'\subseteq T^*L$ is a graph. Note that it is, however, still a geodesic since the path can be generated by an autonomous Hamiltonian (see~\cite{IriyehOtofuji2007}). Therefore, it is always \emph{locally} minimizing. \par

In the case when $L$ is either exact or monotone, we can use the results of Subsections~\ref{subsec:connect} and~\ref{subsec:geodesics} to give a lower estimate on how much the path $t\mapsto \Psi(tL')$ is far away from being minimizing. More precisely, we prove the following.

\begin{cor} \label{cor:lower-bound_hofer-geo}
	Let $L$ be either exact or monotone Lagrangian submanifold of $M$, and let $g$ be a metric on $M$ which corresponds to the Sasaki metric of $L$ on a Weinstein neighbourhood $\Psi:D^*_r L\to M$. For every $k\geq 1$, there are constants $C>0$ and $r'\in(0,r]$ with the following property. Whenever $f:L\to\R$ is such that $\Psi(\operatorname{graph} df)\in\mathscr{L}_k$ and $|df|\leq r'$, we have that
	\begin{align*}
		d_H(\Psi(t\operatorname{graph} df),\Psi(s\operatorname{graph} df))\geq C (t-s)^2 \max|df|^2
	\end{align*}
	for every $t,s\in [0,1]$, where $d_H$ is the Hofer distance in $M$.
\end{cor}
\begin{pr*}
	Note that $\Psi(t\operatorname{graph} df)\in\mathscr{L}_k$ for all $t\in [0,1]$ by Proposition~\ref{prop:path_connect_sasaki}. Therefore, the bound follows directly from Theorem~\ref{thm:first} by using Proposition~\ref{prop:radial_distance} to compute $\delta_H(\Psi(t\operatorname{graph} df),\Psi(s\operatorname{graph} df))$.
\end{pr*}

\begin{rem} \label{rem:lower-bound_hofer-geo}
	In \cite{Chasse2022}, the precise $C$ is computed in terms of $k$ and the sectional curvature and injectivity radius of $M$. In fact, by choosing $r'$ small enough, $C$ can be made to only depend on the values of these invariants on $\Psi(D^*_r L)$. However, those values still depend on more than just $k$~---~except when $L$ is flat. For example, the sectional curvature and the injectivity radius in the Sasaki metric are not uniformly bounded in $T^*L$ when $L$ is not flat (c.f.~\cite{Kowalski1971}), so that $C$ must depend heavily on $r$. On the other hand, we can replace the dependency of $C$ on $k$ for one depending on $||d\phi||$, where $\phi$ is the Hamiltonian diffeomorphism generated by $-\pi^*f$ and $\pi:T^*L\to L$ is the natural projection (see~\cite{ChasseLeclercq2023}).
\end{rem}

\section{Properties in the limit} \label{sec:limit}
There is a natural question of whether the properties of the $\mathscr{L}^\star_k$'s survive in $\mathscr{L}^\star_\infty$. This is however not a simple matter to see which properties can be transported to the limit, as the sequence $\{L_i\}\subseteq \mathscr{L}^e_\infty(T^*S^1)$ in Figure~\ref{fig:bad-sequence_hofer} exemplifies. \par

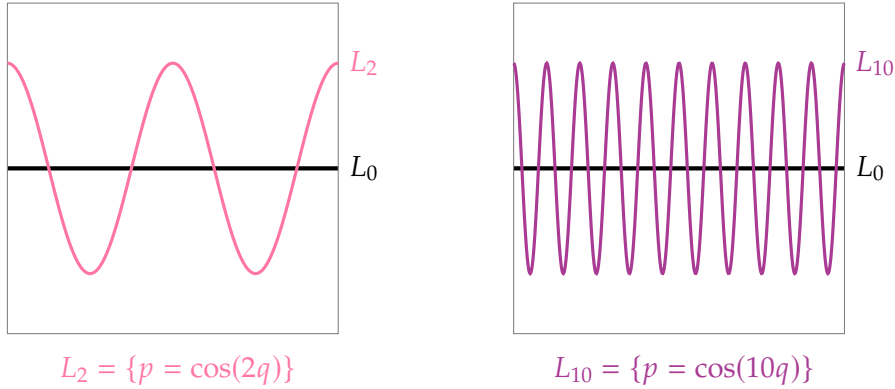
\begin{figure}[ht]
	\centering
	\begin{tikzpicture}[scale=0.7]
		\draw[-,ultra thick] (-pi, 0) -- (pi, 0) node[right] {$L_0$};
		\draw[gray,thin] (-pi,-pi) rectangle (pi,pi);
		\draw[domain=-pi:pi,samples=300,FrenchPink,very thick] plot(\x,{2*cos(2*\x r)}) node[right] {$L_2$};
		\node[FrenchPink] at (0.1,-pi-0.7) {$L_2=\{p=\cos (2q)\}$};
	\end{tikzpicture}
	\qquad\qquad
	\begin{tikzpicture}[scale=0.7]
		\draw[-,ultra thick] (-pi, 0) -- (pi, 0) node[right] {$L_0$};
		\draw[gray,thin] (-pi,-pi) rectangle (pi,pi);
		\draw[domain=-pi:pi,samples=500,Mulberry,very thick] plot(\x,{2*cos(10*\x r)}) node[right] {$L_{10}$};
		\node[Mulberry] at (0.1,-pi-0.7) {$L_{10}=\{p=\cos (10q)\}$};
	\end{tikzpicture}
	\caption{The sequence $\{L_i\}$ and its Hofer limit $L_0$}
	\label{fig:bad-sequence_hofer}
\end{figure}

Indeed, one can easily convince oneself that such a sequence Hofer-converges but does not Hausdorff-converge to $L_0$ (see~\cite{Chasse2023} for a more detailed analysis of this example). This sequence suggests that $\mathscr{L}^\star_k$ is a fairly pathological subspace of $\mathscr{L}^\star_\infty$, since it indicates that every open subset of $\mathscr{L}^\star_\infty$ intersect all $\mathscr{L}^\star_k$ with large $k$. In particular, $\{\mathscr{L}^\star_k\}_{k\geq 1}$ is far from being an exhaustion of $\mathscr{L}^\star_\infty$ by compact sets, which complicates things. \par

Nonetheless, we make here an attempt to extract properties. More precisely, Subsections~\ref{subsec:limit_topo-prop} and~\ref{subsec:limit_classes} are dedicated to proving the two corollaries in the introduction which follow the second principle. That is, in the first subsection, we show that $\mathscr{L}^\star_\infty$ is separable and, in the second one, that it contains at most countably many Hamiltonian isotopy classes. We end this part with Subsection~\ref{subsec:other-limit}, which is a study of another possible limit space of the $\mathscr{L}^\star_k$'s. Even though that space is better suited to the $\mathscr{L}^\star_k$ spaces, we show that it basically has the $C^2$ topology, and thus is far too rigid from the symplectic point of view. \par

\subsection{Topological and metric properties} \label{subsec:limit_topo-prop}
We investigate the implications of Section~\ref{sec:topo} to the space $\mathscr{L}^\star_\infty$ of all Lagrangian submanifolds in $M$ respecting $\star$ when that space is equipped with the metric $d$. This corresponds to the second corollary of Theorem~\ref{thmout:compact-comp} in the introduction. \par

\begin{prop} \label{prop:L_infty-separable}
	The metric space $(\mathscr{L}^\star_\infty,d)$~---~and thus its completion~---~is separable.
\end{prop}
\begin{pr*}
	We first note that every totally bounded metric space $X$ is separable. Indeed, for every $m\geq 1$, we can cover $X$ by a finite number of balls of radius $\frac{1}{m}$. Let $x_1,\dots, x_{N_m}$ be the center of these balls. By construction,
	\begin{align*}
		\mathcal{B}:=\bigcup_{m=1}^\infty \{x_1,\dots, x_{N_m}\}
	\end{align*}
	is then a countable dense subset of $X$. \par
	
	In particular, $\mathscr{L}^\star_k$ admits a countable dense subset $\mathcal{B}_k$ for all $k\geq 1$ by Corollary~\ref{cor:totally_bnd}. Therefore, $\cup_k \mathcal{B}_k$ is the required countable dense subset of $\mathscr{L}^\star_\infty=\cup_k\mathscr{L}^\star_k$ 
\end{pr*}

Owing to the equivalence of many topological properties on metric spaces, we directly get the following. \par

\begin{cor} \label{cor:L_infty-separable-consequences}
	The metric space $(\mathscr{L}^\star_\infty,d)$ and its completion are second countable, paracompact, and hereditarily Lindel\"of, i.e.\ for every subspace $A$, an open cover of $A$ admits a countable subcover.
\end{cor}

\begin{rem} \label{rem:separable_new}
	It was pointed out to us by Vincent Humili\`ere and Egor Shelukhin that the Hamiltonian orbit of any Lagrangian $L$ is always separable in the Lagrangian Hofer metric $d_H$. Indeed, this follows from the fact that $C^\infty_c([0,1]\times M)$ is separable in the $C^1$-norm and that $d_H(\phi^1_H(L),\phi^1_G(L))\leq ||\overline{H}\# G||_H\leq C||H-G||_{C^1}$. Therefore, Proposition~\ref{prop:L_infty-separable} is more of a statement on the behaviour of the Hamiltonian isotopy classes of $\mathscr{L}^\star_\infty$. We will explore them more in depth below.
\end{rem}

\subsection{Symplectic properties} \label{subsec:limit_classes}
We now explore the possible Hamiltonian isotopy classes of $\mathscr{L}^\star_\infty$. This corresponds to the third corollary of Theorem~\ref{thmout:finite_classes} in the introduction. \par

\begin{prop} \label{prop:countably_isotopy_classes}
	There are at most countably many Hamiltonian isotopy classes in $\mathscr{L}^\star_\infty$.
\end{prop}
\begin{pr*}
	We construct a sequence which enumerates all Hamiltonian isotopy classes as follows. For $k=1$, let $\{L_1,\dots,L_{N_1}\}$ be a collection of Lagrangian submanifolds such that $L_i$ and $L_j$ are not Hamiltonian isotopic if $i\neq j$ and such that any $L\in\mathscr{L}^\star_1$ is Hamiltonian isotopic to one of the $L_i$. By Proposition~\ref{prop:finitely_isotopy_classes}, such a $N_1<\infty$ exists. Then, $\{L_1,\dots,L_{N_{k+1}}\}$ is built from $\{L_1,\dots,L_{N_k}\}$ by adding representatives of the Hamiltonian isotopy classes of $\mathscr{L}^\star_{k+1}-\mathscr{L}^\star_k$ is a similar fashion. Since every $L\in\mathscr{L}^\star_\infty$ must be contained in $\mathscr{L}^\star_k$ for some $k$, it is clear that $\{L_1,L_2,\dots\}$ is in bijection with the Hamiltonian isotopy classes in $\mathscr{L}^\star_\infty$.
\end{pr*}

Since every Hamiltonian isotopy class is contained in a single path-connected component of $\mathscr{L}^\star_\infty$, we also get the following.

\begin{cor} \label{cor:countably_comp}
	The space $(\mathscr{L}^\star_\infty,d)$ has at most countably many path-connected components.
\end{cor}

\subsection{Another limit space} \label{subsec:other-limit}
As noted above, we sadly lose many properties of the $\mathscr{L}^\star_k$ spaces when we go to the limit space $\mathscr{L}^\star_\infty$. This is ultimately because the $\mathscr{L}^\star_k$ spaces are quite pathological in $\mathscr{L}^\star_\infty$. There is however another natural topology on the set $\cup_k \mathscr{L}^\star_k$ which circumvents this issue: the limit $\varinjlim \mathscr{L}^\star_k$ of the inductive system $\mathscr{L}^\star_1\subseteq\mathscr{L}^\star_2\subseteq\dots$. In other words, $\varinjlim \mathscr{L}^\star_k=\cup_k \mathscr{L}^\star_k$ as a set, and a subset $U\subseteq \cup_k \mathscr{L}^\star_k$ is open if and only if $U\cap\mathscr{L}^\star_k$ is open in $\mathscr{L}^\star_k$ for all $k$. In particular, this means that $\varinjlim \mathscr{L}^\star_k=\mathscr{L}^\star_\infty$ as sets, but the topology on the metric spaces $(\mathscr{L}^\star_\infty,d)$ and $(\mathscr{L}^\star_\infty,\delta_H)$ is coarser than that of $\varinjlim \mathscr{L}^\star_k$. Note that by Lemmata~\ref{lem:curvature_comparision} and~\ref{lem:tameness_comparision} below, the topology on $\varinjlim \mathscr{L}^\star_k$ is independent of the choice of Riemannian metric. \par

To exemplify how this topology is better behaved in some regards, we show that its connected components are much simpler than those of $\mathscr{L}^\star_\infty$. \par

\begin{prop} \label{prop:comp_lim-L}
	The connected components and path-connected components of $\varinjlim \mathscr{L}^\star_k$ agree, and they are precisely the Hamiltonian isotopy classes.
\end{prop}
\begin{pr*}
	First note that a given Hamiltonian isotopy class is always contained in a single path-connected component of $\varinjlim \mathscr{L}^\star_k$. To see this, suppose that $L, L'\in\varinjlim \mathscr{L}^\star_k$ are Hamiltonian isotopic, and take a Hamiltonian isotopy $\{\phi_t\}$ such that $\phi_1(L)=L'$. By smoothness of the isotopy, the path $c(t)=\phi_t(L)$ is fully contained in $\mathscr{L}^\star_k$ for some $k$ and is Hausdorff-continuous. Therefore, $L$ and $L'$ are in the same path-connected component of $\mathscr{L}^\star_k$, and thus of $\varinjlim \mathscr{L}^\star_k$. \par
	
	On the other hand, given $L\in\varinjlim L_k$, its connected component in $\varinjlim L_k$ must contain the Hamiltonian isotopy class $\mathscr{L}_\infty^L$ of $L$. Indeed, in Corollary~\ref{cor:nearby_conj}, we have shown that $\mathscr{L}^L_k=\mathscr{L}_\infty^L\cap\mathscr{L}^\star_k$ is clopen in $\mathscr{L}^\star_k$ for all $k$, so that $\mathscr{L}_\infty^L$ must also be clopen in $\varinjlim L_k$. But a clopen must contain the connected component of its elements, which proves the inclusion. \par
	
	Since each path-connected component is contained in a single connected component, this proves the result.
\end{pr*}

We now compare the limit topology with other ones to better understand it. The following example shows that these topologies are strictly coarser than the limit topology. Note that the example is for $M=T^*S^1$, but it can easily be generalized to any symplectic manifold by using a Darboux chart adapted to a given Lagrangian submanifold. \par

\begin{ex*}
	On $M=T^*S^1$, consider the 1-parameter family of Hamiltonian $\{H^s\}_{s>0}$ defined via
	\begin{align*}
		H^s(q,p)=s^{3/2}\beta(q)\sin\left(\frac{q}{s}\right),
	\end{align*}
	where we have identified $S^1$ with $\R/\Z$ and $\beta:[0,1]\to [0,1]$ is zero near $\{0,1\}$ and takes value 1 on $[\frac{1}{4},\frac{3}{4}]$. Set $L^0=\{p=0\}$ and $L^s=\phi^{H^s}_1(L^0)=\graph (-dH^s)$. \par
	
	It is easy to see that $s\mapsto L^s$, $s\in [0,1]$, defines a path which is continuous with respect to both $\delta_H$ and $d_H$ (and thus also $d$). In particular, the set $\{L^s\}_{s\in [0,1]}\subseteq \mathscr{L}^e_\infty$ is compact in all these metrics. On the other hand, it is not contained in any $\mathscr{L}^e_k$, since $\lim_{s\to 0}||B_{L^s}||=\infty$. Therefore, it cannot be compact in the limit topology (see Lemma~\ref{lem:limit_compacts} below).
\end{ex*}

\begin{lem} \label{lem:limit_compacts}
	Let $\{X_k\subseteq X\}$ be an increasing sequence of compact subspaces of a Hausdorff space $X$. The space $\varinjlim X_k$ is Hausdorff. Moreover, a subset $A$ of $\varinjlim X_k$ is compact if and only if it is closed and $A\subseteq X_k$ for some $k$.
\end{lem}
\begin{pr*}
	We first prove the Hausdorffness. If $x\neq y\in \varinjlim X_k=\cup_k X_k$, then there are open subsets $U$ and $V$ of $X$ such that $x\in U$ and $y\in V$, but $U\cap V=\emptyset$. But then, the restrictions $U\cap X_k$ and $V\cap X_k$ are open for each $k$, so that they are also open in $\varinjlim X_k$. Thus, $x$ and $y$ are also separated in $\varinjlim X_k$, and $\varinjlim X_k$ is indeed Hausdorff. \par
	
	We now prove the equivalence. One direction is obvious. Let thus $A$ be compact in $\varinjlim X_k$. Since the limit space is Hausdorff, $A$ must be closed. Suppose however that $A$ is not contained in any $X_k$. Then, for every $k$, there is some $x_k\in A-X_k$. In particular, the set $S=\{x_k\}\subseteq A$ is such that $S\cap X_k$ is finite for all $k$. Therefore, that intersection is closed, since $X_k$ is Hausdorff. By definition of the limit topology, this thus means that $S$ itself is closed. In fact, this logic shows that every subset of $S$ is closed, i.e.\ $S$ is a closed infinite discrete subset of $A$. But this is impossible if $A$ is compact, hence the contradiction.
\end{pr*}

The above lemma allows us to completely characterize the limit topology. \par

\begin{prop} \label{prop:charac_converging-sequences_limit}
	A sequence $\{L_i\}\subseteq\varinjlim\mathscr{L}^\star_k$ converges to some $L\in \varinjlim\mathscr{L}^\star_k$ if and only if there exists some $k\in\N$ such that $\{L_i\}\subseteq\mathscr{L}^\star_k$ and $L_i\to L$ in the Hausdorff topology~---~or in any of the many equivalent topology on $\mathscr{L}^\star_k$.
\end{prop}

Indeed, the result is a direct consequence of Lemma~\ref{lem:limit_compacts} just above and of the following simple fact from point-set topology. \par

\begin{lem} \label{lem:converging-sequences_compact}
	Let $\{x_i\}_{i\in\N}$ be a sequence in a Hausdorff space $X$ such that none of its elements is a limit point. The sequence converges to some $x\in X$ if and only if the subspace $\overline{S}=\{x_i\}_{i\in\N}\cup\{x\}$ is compact.
\end{lem}
\begin{pr*}
	Suppose that $x_i\to x$, and let $\{U_a\}_{a\in A}$ be an open cover of $\overline{S}$. Then, there is some $a_0\in A$ such that $x\in U_{a_0}$. But by convergence, there is some $N\in\N$ such that $x_i\in U_{a_0}$ for all $i> N$. It then suffices to pick $a_i$ such that $x_i\in U_{a_i}$ for each $i\leq N$ to get a finite subcover $\{U_{a_i}\}_{i=0}^N$.
	
	Suppose now that $\overline{S}$ is compact, and let $U_0$ be an open neighbourhood of $x$. Note that for each $x_i$, there is some open $U_i$ such that $U_i\cap\overline{S}$ is finite. Indeed, otherwise, $x_i$ would be a limit point of the sequence, which would be a contradiction with the hypothesis on $\{x_i\}$. Since $X$ is Hausdorff, we may suppose that $U_i\cap\overline{S}=\{x_i\}$. Therefore, there must be only a finite number of $i$ such that $x_i\notin U_0$, otherwise $\{U_i\}_{i=0}^\infty$ would be an open cover of $\overline{S}$ with no finite subcover.
\end{pr*}

Given Proposition~\ref{prop:charac_converging-sequences_limit}, we can see where the limit topology sits with regard to the various $C^k$-topologies. \par

\begin{cor} \label{cor:limit-topo_versus_C^k}
	The limit topology is (strictly) finer than the $C^{1,\alpha}$-topology, for any $0<\alpha<1$, but coarser than the $C^2$-topology.
\end{cor}
\begin{pr*}
	By Theorem~\ref{thm:second}, the $C^{1,\alpha}$-topology is equivalent to the Hausdorff topology on $\mathscr{L}^\star_k$ for any $0<\alpha<1$. Therefore, every convergent sequence in the limit topology is also convergent (with the same limit) in the $C^{1,\alpha}$-topology. To see that the inclusion of topology is strict, just use the example above, but replace $s^{3/2}$ by $s^{2+\alpha}$ in the definition of $H^s$.
	
	The $C^2$-topology is coarser because every $C^2$-converging sequence has uniformly bounded second fundamental form (direct computation), are uniformly $\epsilon$-tame (this is the idea of Lemma~\ref{lem:tameness_2d}), and obviously also converge in the Hausdorff topology to the same limit.
\end{pr*}

\vfill

\pagebreak
\appendix
\renewcommand{\thesection}{{A\Alph{section}}}
\renewcommand{\thesubsection}{A\Alph{section}.\arabic{subsection}}
\titleformat{\section}[block]{\large\scshape}{Appendix \Alph{section}.}{0.5em}{}
\titleformat{\subsection}[block]{\scshape}{\Alph{section}.\arabic{subsection}.}{0.5em}{}

\section{Some results in Riemannian geometry} \label{sec:riem-geo}
This appendix compiles all the results in Riemannian geometry which were required throughout the paper, but that the author could not find in the literature. We suspect that many of these results are known to expect, but this is not the case for all of them: at least Lemma~\ref{lem:complete_sasaki} has appeared as a conjecture in a paper of Albuquerque~\cite{Albuquerque2019}. Below, Subsection~\ref{subsec:geo_sasaki} compiles the results on the Sasaki metric, Subsection~\ref{subsec:geo_2d} on Riemannian surfaces, and Subsection~\ref{subsec:geo_comparison} on comparison results between Riemannian invariants of Lagrangian submanifolds. \par

\subsection{Results on the Sasaki metric} \label{subsec:geo_sasaki}
\subsubsection*{Behavior along graphs}
We begin by proving some useful results on the behaviour of some Riemannian invariants of graphs in $TL$ of vector fields under the transformation $(x,v)\mapsto (x,tv)$. More precisely, we show that the norm of the second fundamental form and the tameness constant must the nondecreasing in $t$ if the vector field is (locally) a gradient of a $C^2$-small function. \par

We begin by studying the norm of the curvature. The proof is elementary but still subtle. \par

\begin{lem} \label{lem:curvature}
	Equip $L$ with a Riemannian metric $g=\langle\cdot,\cdot\rangle$ and $TL$ with its associated Sasaki metric. Let $\xi=\operatorname{grad}H\in\mathfrak{X}(L)$. If $|\xi|$ and $|\nabla\xi|$ are sufficiently small, then the function $t\mapsto ||B_{t\xi}||$ is nondecreasing for $t\in [0,1]$.
\end{lem}
\begin{pr*}
	For $X\in T_xL$, we denote by $X^h$ and $X^v$ its horizontal and vertical lifts in $T_{(x,y)}TL$, respectively. Then, we have that
	\begin{align}
		T_{\xi(x)}\xi(L)&=\left\{\widetilde{X}=X^h+(\nabla_X\xi)^v\ \middle|\ X\in T_xL\right\} \label{eqn:tgt}
		\intertext{and}
		T_{\xi(x)}^\perp\xi(L) &=\left\{\widetilde{Z}=Z^v-((\nabla\xi)^*Z)^h\ \middle|\ Z\in T_xL\right\}, \label{eqn:perp}
	\end{align}
	where $\langle (\nabla\xi)^*Z,Y\rangle=\langle Z,\nabla_Y\xi\rangle$ for all $Y\in T_xL$ (see for example~\cite{AbbassiYampolsky2004}). Denoting by $\widetilde{\nabla}$ the Levi-Civita connection on $TL$, by $\nabla$ the Levi-Civita connection on $L$, and by $R$ the Riemann curvature tensor on $L$, we get that
	\begin{align*}
		\widetilde{\nabla}_{X^h}Y^h &=(\nabla_XY)^h-\frac{1}{2}(R(X,Y)\xi)^v, \\
		\widetilde{\nabla}_{X^h}Y^v &=(\nabla_XY)^v+\frac{1}{2}(R(\xi,Y)X)^h, \\
		\widetilde{\nabla}_{X^v}Y^h &=\frac{1}{2}(R(\xi,X)Y)^h, \\
		\widetilde{\nabla}_{X^v}Y^v &=0, \\
	\end{align*}
	for all $X,Y\in\mathfrak{X}(L)$. Therefore, the expression for the second fundamental form of $\xi(L)$ is
	\begin{align*}
		B_\xi\left(\widetilde{X},\widetilde{X},\widetilde{Z}\right)=\overbrace{\left\langle \nabla^2_X\xi-\nabla_{\nabla_X X}\xi,Z\right\rangle}^\alpha-\overbrace{\left\langle \nabla_{R(\xi,\nabla_X \xi)X}\xi,Z\right\rangle}^\beta,
	\end{align*}
	for all $\widetilde{X}\in T_{\xi(x)}\xi(L)$ and all $\widetilde{Z}\in T_{\xi(x)}^\perp\xi(L)$. \par
	
	From (\ref{eqn:tgt}) and the definition of the Sasaki metric, we have that $|\widetilde{X}|^2=|X|^2+|\nabla_X\xi|^2$. Likewise, from (\ref{eqn:perp}), we have that $|\widetilde{Z}|^2=|Z|^2+|\nabla_Z\xi|^2$; this is because
	\begin{align*}
		\left|(\nabla\xi)^*Z\right|^2
		=\operatorname{Hess}H\left(Z,(\nabla\xi)^*Z\right)
		=\left\langle(\nabla\xi)^*Z,\nabla_Z\xi\right\rangle
		=\operatorname{Hess}H\left(Z,\nabla_Z\xi\right)
		=\left|\nabla_Z\xi\right|^2,
	\end{align*}
	since the Hessian of a function is symmetric. Therefore, for every $t\in [0,1]$, the map
	\begin{center}
		\begin{tikzcd}[row sep=0pt,column sep=1pc]
			TTL \arrow{r} & TTL \\
			{\left((x,y),\widetilde{Y}\right)} \arrow[mapsto]{r} & {\left((x,ty),(1+(t^2-1)|\nabla_Y\xi|^2)^{-1/2}\widetilde{Y}\right)}
		\end{tikzcd}
	\end{center}
	sends $\xi(x)$ to $t\xi(x)$ and sends diffeomorphically the unit sphere of $T_{\xi(x)}\xi(L)$, respectively of $T_{\xi(x)}^\perp\xi(L)$, onto the one of $T_{t\xi(x)}t\xi(L)$, respectively of $T_{t\xi(x)}^\perp t\xi(L)$. Here, $Y$ denotes the sum of the projections of $\widetilde{Y}$ onto the horizontal and vertical distributions, after their identification with $TL$. Note that, on these spheres, $|\widetilde{Y}|^2=|Y|^2+|\nabla_Y\xi|^2=1$ so that $|Y|\leq 1$ and $|\nabla_Y\xi|<1$. In particular, $s_Y(t):=(1+(t^2-1)|\nabla_Y\xi|^2)^{-1/2}$ is well defined for $t\in [0,1]$. \par
	
	Therefore, it does suffice to prove that the map
	\begin{equation} \label{eqn:fcn_to_min}
		\begin{tikzcd}[row sep=0pt,column sep=1pc]
			t \arrow[mapsto]{r} & \left|\left(B_{t\xi}(s_X(t)\widetilde{X},s_X(t)\widetilde{X},s_Z(t)\widetilde{Z})\right)\right|=s_X^2s_Z t|\alpha-t^2\beta|
		\end{tikzcd}
	\end{equation}
	is nondecreasing for $t\in [0,1]$, for all $\widetilde{X}\in T_{\xi(x)}\xi(L)$ and all $\widetilde{Z}\in T_{\xi(x)}^\perp\xi(L)$ such that $|\widetilde{X}|=|\widetilde{Z}|=1$. Indeed, by the previous discussion, this will imply that the map $t\mapsto |B_{t\xi}|^{op,sym}_x$~---~sending $t$ to the operator norm of $B_{t\xi}$ on the subspace of $T_{\xi(x)}t\xi(L)\otimes T_{\xi(x)}t\xi(L)\otimes T_{\xi(x)}^\perp t\xi(L)$ generated by elements of the form $\widetilde{X}\otimes\widetilde{X}\otimes\widetilde{Z}$~---~is nondecreasing. Since $B_{t\xi}$ is symmetric in its first two entries, this is just the operator norm on the whole space, and we will get the result. \par
	
	If $\alpha=0$, we may suppose that $\beta\neq 0$, otherwise (\ref{eqn:fcn_to_min}) is just the zero function, and the statement is trivial. In that case, (\ref{eqn:fcn_to_min}) looks like $6|\beta|(1-|\nabla_X\xi|^2)^{-1}(1-|\nabla_Z\xi|^2)^{-1/2}t^3+\mathcal{O}(t^4)$ near $t=0$. In particular, it is increasing near $t=0$. But (\ref{eqn:fcn_to_min}) only possibly has critical points at $t=0$ and $t=\pm\sqrt{\frac{3(1-|\nabla_X\xi|^2)(1-|\nabla_Z\xi|^2)}{|\nabla_X\xi|^2|\nabla_Z\xi|^2-|\nabla_X\xi|^2|-2|\nabla_Z\xi|^2}}$. For $|\nabla\xi|$ small enough, the latter values are not real, and thus (\ref{eqn:fcn_to_min}) is increasing. \par
	
	Suppose now that $\alpha\neq 0$. By changing the sign of $\widetilde{Z}$ if necessary, we may assume that $\alpha>0$. Since $|\beta|\leq ||R||\ |\xi||\nabla\xi|^2$ and $\alpha$ depends only on derivatives of $\xi$, we thus have that $|\alpha-t^2\beta|=\alpha-t^2\beta$ for all $t\in [0,1]$ if $|\xi|$ is small enough. But the function $t\mapsto s_X^2s_Z t(\alpha-t^2\beta)$ converges with all derivatives to the function $t\mapsto t\alpha$ as $|\nabla\xi|\to 0$. Since that function is increasing, the derivative of (\ref{eqn:fcn_to_min}) is positive for all $t\in [0,1]$ for $|\nabla\xi|$ small enough. Therefore, the function is increasing over the interval for $|\xi|$ and $|\nabla\xi|$ small enough.
\end{pr*}

\begin{rem} \label{rem:curvature_small}
	Given the proof of Lemma~\ref{lem:curvature}, it appears that how small we must take $|\xi|$ and $|\nabla\xi|$ depends on $\widetilde{X}$ and $\widetilde{Z}$. However, since the infimum to get the operator norm $||B_{t\xi}||$ is taken over the unit sphere, which is compact, it is in fact a minimum. Therefore, how small we take $|\xi|$ and $|\nabla\xi|$ can be made independent of $\widetilde{X}$ and $\widetilde{Z}$.
\end{rem}

We now move on to studying the tameness constant. This time, the proof is fairly straightforward. \par

\begin{lem} \label{lem:tame_TL}
	Let $\xi\in\mathfrak{X}(L)$, and equip $TL$ with a Sasaki metric. Denote
	\begin{align*}
		\epsilon_\xi:=\inf_{x\neq y\in \xi(L)}\frac{d_{TL}(x,y)}{\min\{1,d_{\xi}(x,y)\}}\in (0,1],
	\end{align*}
	where $d_\xi$ denotes the intrinsic distance in $\xi(L)$. Then,
	\begin{align*}
		\lim_{|\nabla\xi|\to 0}\epsilon_\xi=1.
	\end{align*}
	In particular, $\epsilon_\xi> (k+1)^{-1}$ for $|\nabla\xi|$ small enough.
\end{lem}
\begin{pr*}
	Consider a path $\gamma:[0,\ell]\to L$ such that $|\dot{\gamma}|\equiv 1$. Denote $\widetilde{\gamma}:=\xi\circ\gamma$. Then, $|\dot{\widetilde{\gamma}}|^2=1+|\nabla_{\dot{\gamma}}\xi|^2$, so that
	\begin{align*}
		\ell\leq \int_0^\ell \left|\dot{\widetilde{\gamma}}\right|dt\leq \ell\sqrt{1+|\nabla\xi|^2}.
	\end{align*}
	Taking the infimum of the above inequality over all paths $\gamma$ such that $\gamma(0)=x$ and $\gamma(\ell)=y$ for given $x,y\in L$, we get that
	\begin{align} \label{eqn:d_xi-ineq}
		d_L(x,y)\leq d_\xi(\xi(x),\xi(y))\leq \sqrt{1+|\nabla\xi|^2}d_L(x,y),
	\end{align}
	since all smooth paths $\widetilde{\gamma}$ in $\xi(L)$ may be parametrized to be of the form  $\widetilde{\gamma}=\xi\circ\gamma$ with $|\dot{\gamma}|\equiv 1$. \par
	
	On the other hand, take a path $\widetilde{\gamma}:[0,\widetilde{\ell}]\to TL$ such that $|\dot{\widetilde{\gamma}}|\equiv 1$, $\gamma(0)=\xi(x)$, and $\gamma(L)=\xi(y)$ for given $x,y\in L$. Then, we may write $\dot{\widetilde{\gamma}}=\dot{\gamma}^h+Y^v$ for $\gamma:=\pi\circ\widetilde{\gamma}$ and a vector field $Y$ of $L$ along $\gamma$, where $\pi:TL\to L$ is the canonical projection. We thus get
	\begin{align*}
		\widetilde{\ell}=\int_0^{\widetilde{\ell}}\sqrt{\left|\dot{\gamma}\right|^2+|Y|^2}dt\geq \int_0^{\widetilde{\ell}}\left|\dot{\gamma}\right|dt.
	\end{align*}
	Taking the infimum over all possible $\widetilde{\gamma}$, we get that
	\begin{align} \label{eqn:d_TL-ineq_left}
		d_{TL}(\xi(x),\xi(y))\geq d_L(x,y),
	\end{align}
	since every path in $L$ from $x$ to $y$ admits a lift to $TL$ from $\xi(x)$ to $\xi(y)$ (e.g.\ $\xi\circ\gamma$). \par
	
	Putting (\ref{eqn:d_xi-ineq}) and (\ref{eqn:d_TL-ineq_left}) together, we thus get
	\begin{align*}
		\inf_{x\neq y}\frac{d_L(x,y)}{\min\{1,d_L(x,y)\sqrt{1+|\nabla\xi|^2}\}}\leq\epsilon_\xi\leq 1,
	\end{align*}
	which implies the result.
\end{pr*}

\begin{rem} \label{rem:curvature-vs-tame}
	The approaches in the proofs of Lemmata~\ref{lem:curvature} and~\ref{lem:tame_TL} are different, because $||B_\xi||$ depends on higher derivatives of $\xi$, while $\epsilon_\xi$ does not. Therefore, just taking a limit in the expression for $||B_{\xi}||$ would not lead to 0, and we could not conclude anything. We need to actually understand the behavior of $||B_{\xi}||$ in a $C^{1,\alpha'}$-neighbourhood of 0, not just in the limit $|\xi|\to 0$.
\end{rem}

\subsubsection*{Behavior of the geodesics}
To adapt Lemma~\ref{lem:tame_TL} to a metric which is only locally Sasaki~---~which is required in Subsection~\ref{subsec:connect}~---~we will need the following technical results. As far as we know, these results have not appeared in the literature. This is also why Lemma~\ref{lem:complete_sasaki} is established in such generality: it has appeared before as a conjecture of Albuquerque~\cite{Albuquerque2019} and could be of general interest to the Riemannian geometry community. \par

\begin{lem} \label{lem:complete_sasaki}
	The space $TN$ is complete in the Sasaki metric associated with $(N,g)$ if and only if $(N,g)$ is complete.
\end{lem}
\begin{pr*}
	One direction is obvious: if $TN$ is complete and $v\in T_x N$, the exponential of $tv\in T_x N\subseteq T_x TN$ exists in $TN$ for all $t\in\R$. However, $N$ is totally geodesic in $TN$~\cite{Sasaki1958}, so that the geodesic $t\mapsto\exp_x^{TN}(tv)$ must stay in $N$. It is thus the exponential of $tv$ in $N$, and $N$ must be complete. \par
	
	Suppose now that $N$ is complete. We recall that completeness of a Riemannian manifold $M$ is equivalent to it having an exhaustion by compact sets $\{K_i\}$ such that, if $\{x_i\}$ is a sequence with $x_i\notin K_i$, then $d(y,x_i)\to\infty$ for some point $y\in M$~---~this is part of the classical Hopf-Rinow theorem, see for example Theorem~7.2.8 in~\cite{doCarmo1992}. \par 
	
	Let $y\in N$ and $K_i=D_i N|_{B_i^N(y)}$, where $B_i^N(y)$ is the (closed) ball of radius $i$ in $N$ centered at $y$. Let $\{x_i\}$ such that $x_i\notin K_i$. We now study two possible types of subsequences of $\{x_i\}$.
	\begin{enumerate}[label=(\arabic*)]
		\item Suppose there is a subsequence, still denoted $\{x_i\}$, such that none of the subsequences of $\{\pi(x_i)\}$ are contained in any of the $B_i^N(y)$. Then, the sequence of natural numbers given by
		\begin{align*}
			n_i=\min\left\{j\ \middle|\ \pi(x_i)\in B_j^N(y)\right\}
		\end{align*}
		converges to infinity. By completeness of $N$, this must mean that $d(y,\pi(x_i))\to\infty$. But since $\pi$ is a Riemannian submersion, it is nonexpansive in $d$, so that
		\begin{align*}
			d(y,\pi(x_i))=d(\pi(y),\pi(x_i))\leq d(y,x_i).
		\end{align*}
		Therefore, $d(y,x_i)\to\infty$. \par
		
		\item Suppose there is a subsequence, still denoted $\{x_i\}$, and a $R>0$ such that $\{\pi(x_i)\}$ is contained in $B_R^N(y)$. Since $x_i\notin K_i$, this forces that $x_i\notin D_i N$ for large enough $i$. Then, let $\gamma_i$ be a minimal geodesic of $N$ from $\pi(x_i)$ to $y$. Let $x'_i:=P_{\gamma_i}(x_i)$ be the parallel transport of $x_i$ along $\gamma_i$. Note that the horizontal lift $\widetilde{\gamma}_i$ of $\gamma_i$ starting at $x_i$ ends at $x'_i$ by construction. Therefore, we have that
		\begin{align*}
			d(x_i,x'_i)\leq \ell(\widetilde{\gamma}_i)=\ell(\gamma_i)=d(\pi(x_i),y)\leq R.
		\end{align*}
		On the other hand, since parallel transport is an isometry on the fibres, the fact that $x_i\notin D_i N$ ensures that $x'_i\notin D_i N$. Since $x'_i$ is in the fibre over $y$, this thus implies that $d(y,x'_i)> i$. Therefore, the triangle inequality gives that
		\begin{align*}
			d(y,x_i)\geq d(y,x'_i)-d(x_i,x'_i)> i-R,
		\end{align*}
		and $d(y,x_i)\to\infty$.
	\end{enumerate}
	
	Since the original $\{x_i\}$ sequence can be written as the union of subsequences of either type, we conclude that $d(y,x_i)\to\infty$, so that $TN$ is complete.
\end{pr*}

\begin{lem} \label{lem:geo_sasaki}
	If $\alpha:[0,\ell]\to TL$ is a geodesic in the Sasaki metric, then the function $t\mapsto |\alpha(t)|^2$ is either constant or a (strictly) convex parabola. In particular, the disk bundle $D_rL$ of radius $r$ is geodesically convex.
\end{lem}
\begin{pr*}
	See $\alpha$ as a vector field $Y$ along the path $x:=\pi\circ\alpha:[0,\ell]\to L$. Note that $|\alpha|=|Y|$. Then, the geodesic on $TL$ is equivalent to two equations on $L$~\cite{Sasaki1958}:
	\begin{align*}
		\begin{cases}
			\nabla_{\dot{x}}\dot{x}+R(Y,\nabla_{\dot{x}}Y)\dot{x}=0; \\
			\nabla_{\dot{x}}^2Y=0.
		\end{cases}
	\end{align*}
	But then, this means that
	\begin{align*}
		\frac{d}{dt}|Y|^2&=2\langle \nabla_{\dot{x}} Y,Y\rangle; \\
		\frac{d^2}{dt^2}|Y|^2&=2\langle \nabla_{\dot{x}}^2 Y,Y\rangle+2|\nabla_{\dot{x}}Y|^2=2|\nabla_{\dot{x}}Y|^2; \\
		\frac{d}{dt}|\nabla_{\dot{x}}Y|^2&=2\langle\nabla_{\dot{x}}^2 Y,\nabla_{\dot{x}} Y\rangle=0.
	\end{align*}
	From the last equation, we get that $|\nabla_{\dot{x}}Y|$ is independent of time. If $|\nabla_{\dot{x}}Y|\equiv 0$, then $\nabla_{\dot{x}}Y\equiv 0$, so that the first equation implies that $|Y|$ is constant. If $|\nabla_{\dot{x}}Y|>0$, then the middle equation implies that the second time derivative of $|Y|^2$ is a positive constant, thus giving that it is a strictly convex parabola.
\end{pr*}

As mentioned above, we need to also study metrics which are only locally Sasaki. By this, we mean that $(M,\omega)$ is a symplectic manifold with a compatible almost complex structure $J$ and that $g=\omega(\cdot, J\cdot)$. We also suppose that $L$ is a Lagrangian submanifold of $M$. By locally Sasaki, we mean that there is a diffeomorphism $\Psi$ from a neighbourhood of $L$ in $TL$ to a neighbourhood of $L$ in $M$ such that $\Psi^*g$ is the Sasaki metric. This makes sense since $J$ identifies $TL$ with $TL^\perp\subseteq TM$. \par

In this case, Lemma~\ref{lem:curvature} obviously still applies, but Lemma~\ref{lem:tame_TL} needs to be adapted, as we could have $d_{TL}\neq d_{M}$. In other words, we have to deal with the fact that the minimal geodesic in $M$ between two points of $L'=\xi(L)$ might not be entirely contained in the neighbourhood of $L$ where $g$ is equal to a Sasaki metric. In particular, that minimal geodesic could be shorter than one would expect in $TL$ so that the ratio $d_M/d_L$ might no longer tend to 1 as $|\nabla\xi|\to 0$. We prove that this in fact cannot happen. \par

\begin{lem} \label{lem:dM-to-dTL}
	Let $M$, $L$, and $g$ as above. On a small enough neighbourhood of $L$, we have that $d_{M}=d_{TL}$.
\end{lem}
\begin{pr*}
	We first prove that $d_M(x,y)=d_{TL}(x,y)$ whenever $x,y\in L$. Let thus $x$ and $y$ be in $L$. Let $\alpha:[0,\ell]\to TL$ be a minimal geodesic in $TL$ from $x$ to $y$. Note that it follows from Lemma~\ref{lem:geo_sasaki} that $\alpha$ is fully contained in $L$, so that $d_{TL}(x,y)=d_L(x,y)$. In particular, $\alpha$ is also a geodesic of $M$. Therefore, it is locally minimizing in $M$, and we can take
	\begin{align*}
		\ell':=\sup\{t\in [0,\ell]\ |\ d_M(x,\alpha(s))=s\ \forall s<t\}\in (0,\ell].
	\end{align*}
	Suppose that $\alpha$ is not minimizing, i.e.\ $\ell'<\ell$, and let $\gamma:[0,\ell']\to M$ a minimal geodesic in $M$ from $x$ to $y':=\alpha(\ell')$ which is different from $\alpha$~---~nonminimality of $\alpha$ ensures that it exists. \par
	
	From classical facts from Riemannian geometry (see Proposition~13.2.12 of~\cite{doCarmo1992} for example), exactly one of two things can happen: either there is a 1-parameter family $\gamma_s$ of geodesics from $x$ to $y'$ with $\gamma_0=\alpha$ and $\gamma_1=\gamma$, or $\gamma$ and $\alpha$ are the only two minimizing geodesic from $x$ and $y'$ and $\gamma'(\ell')=-\alpha'(\ell')$. In the first case, note that all $\gamma_s$ have the same length since geodesics are critical points of the length functional. Therefore, for $s$ small enough, $\gamma_s$ is a minimal geodesic not contained in $L$, but fully contained in the neighbourhood of $L$ where $g$ is the Sasaki metric. This is of course a contradiction with Lemma~\ref{lem:geo_sasaki}, since 0 and $\ell$ would then both have to be strict minima of $t\mapsto |\gamma_s(t)|^2$. In the second case, $\gamma$ is then tangent to $L$ at $t=\ell'$. But since $L$ is totally geodesic, $\gamma$ must then be fully contained in $L$, and we again get a contradiction. Therefore, the result holds on $L$. \par
	
	We now consider $x$ and $y$ close to $L$ in $M$. Let $\alpha$, $\gamma$, $\ell'$, and $y'$ be defined analogously as above. From Lemma~\ref{lem:complete_sasaki}, $\alpha$ exists and, if we take $x$ and $y$ to be in a neighbourhood of the form $\Psi(D_r L)$, then $\alpha$ stays in that neighbourhood by Lemma~\ref{lem:geo_sasaki}. Again, we have two possibilities for how $\gamma$ and $\alpha|_{[0,\ell']}$ connect. If $\gamma_s$ is a 1-parameter family, then we still get a contradiction for some $s$: since $\gamma_1$ leaves the Weinstein neighbourhood and $\gamma_s$ always has the same endpoints, there is some $s$ such that $\gamma_s$ is still in that neighbourhood, but such that $t\mapsto |\gamma_s(t)|$ has a maximum (in contradiction with Lemma~\ref{lem:geo_sasaki}). \par 
	
	However, the second possibility~---~that $\alpha$ and $\gamma$ form a geodesic loop in $M$~---~does not \textit{a priori} lead to a contradiction as things are. Thus take sequences $\{x_i\}$ and $\{y_i\}$ such that $\lim_i d_M(x_i,L)=\lim_i d_M(y_i,L)= 0$, but such that $d_M(x_i,y_i)<d_{TL}(x_i,y_i)$ for all $i$. Define $\alpha_i$, $\gamma_i$, $\ell'_i$, and $y'_i$ analogously as before. In particular, if $v_i$ is the unit vector such that $\exp_{x_i}(tv_i)=\alpha(t)$ for $t\in [0,\ell_i]$, then $t\mapsto \exp_{x_i}(tv_i)$, $t\in [2,\ell'_i]$, is the geodesic loop $\alpha_i\#\overline{\gamma_i}$. Since $\{x_i\}$, $\{y_i\}$, and $\{v_i\}$ are all contained in a compact, we may pass to a subsequence, so that $\lim_i x_i= x\in L$, $\lim y_i=y\in L$, and $\lim v_i=v\in T_x M$. But then, $t\mapsto \exp_x(tv)$, $t\in [0,2\ell']$ is a geodesic loop in $M$ which is fully contained in $L$ over $[0,\ell]$, but that eventually leaves it. Therefore, we get a last contradiction, and we must have $d_M(x,y)=d_{TL}(x,y)$ whenever $x$ and $y$ are close to $L$.
\end{pr*}

\subsection{Result on Riemann surfaces} \label{subsec:geo_2d}
We suppose that $(M,g,J,\omega)$ is a Riemann surface. Then, a tubular neighbourhood of a curve $L$ admits some fairly nice coordinates given by
\begin{center}
	\begin{tikzcd}[row sep=0pt,column sep=1pc]
		\phi\colon (0,\ell)\times (-r,r) \arrow{r} & M \\
		{\hphantom{\phi\colon{}}} (s,t) \arrow[mapsto]{r} & \exp_{\gamma(s)}(tJ\dot{\gamma}(s))
	\end{tikzcd},
\end{center}
where $\gamma:[0,\ell)\to L$ is a parametrization such that $|\dot{\gamma}|\equiv 1$. Note that
\begin{align*}
	\phi^*g &=|W|^2ds^2+dt^2 \\
	\text{and}\qquad \phi^*J &= |W|\frac{\del}{\del t}\otimes ds-\frac{1}{|W|}\frac{\del}{\del s}\otimes dt,
\end{align*}
where $W(s,t)$ is the value at time $t$ of the unique Jacobi field along the geodesic $t\mapsto\phi(s,t)$ such that $W(s,0)=\dot{\gamma}(s)$ and $\dot{W}(s,0)=-\kappa(s)\dot{\gamma}(s)$. Here, $\kappa$ is the (signed) geodesic curvature of $L$, which is defined via the relation $\ddot{\gamma}:=\nabla_{\dot{\gamma}}\dot{\gamma}=\kappa J\dot{\gamma}$. \par

In these coordinates, we will call $L'=\{(s,\xi(s))\ |\ s\in [0,\ell)\}=:\operatorname{graph}\xi$ the graph of $\xi:[0,\ell]\to\R$. Because our applications are aimed towards Subsection~\ref{subsec:connect}, we will be interest in 1-parameter families of the form $\{\xi_\alpha=\alpha\xi+c(\alpha)\}_{\alpha\in [0,1]}$, where $||\xi||<\frac{r}{2}$ and $c$ is such that
\begin{enumerate}[label=(\roman*)]
	\item $c(0)=c(1)=0$;
	\item $|c(\alpha)|\leq \alpha||\xi||$.
\end{enumerate}
In particular, $\operatorname{graph}\xi_\alpha$ stays in the chart defined by $\phi$ for all $\alpha\in [0,1]$.

Before moving on with the results on $\xi_\alpha$, we prove the following lemma on the behaviour of $|W|$ for small values of $|t|$, which will be quite useful later on. \par

\begin{lem} \label{lem:taylor_jacobi_2d}
	Let $K:[0,\ell)\to\R$ be the pullback of the Gaussian curvature of $M$ along $\gamma$. We have that
	\begin{align*}
		|W|^2=1-2\kappa t+(\kappa^2-K)t^2+\mathcal{O}(t^3).
	\end{align*}
\end{lem}
\begin{pr*}
	As noted above, we have that $W(s,0)=\dot{\gamma}(s)$ and $\dot{W}(s,0)=-\kappa(s)\dot{\gamma}(s)$. The lemma then follows directly from the following computations:
	\begin{align*}
		\left.|W|^2\ \right|_{t=0} &=1 \\[5pt]
		\left.\frac{\del}{\del t}|W|^2\ \right|_{t=0} &=\left.2\langle \dot{W},W\rangle\right|_{t=0}= -2\kappa \\[5pt]
		\left.\frac{\del^2}{\del t^2}|W|^2\ \right|_{t=0} &=2\left.|\dot{W}|^2\ \right|_{t=0}+\left.2\langle \ddot{W},W\rangle\right|_{t=0} \\
		&=2\kappa^2-2\langle R(J\dot{\gamma},\dot{\gamma})J\dot{\gamma},\dot{\gamma}\rangle = 2(\kappa^2-K),
	\end{align*}
	where the last line follows from the fact that $W$ satisfies the Jacobi equation.
\end{pr*}

With this in hand, we can estimate the geodesic curvature of $\operatorname{graph}\xi_\alpha$ in terms of that of $L$ and $\operatorname{graph}\xi$. More precisely, we want to prove the following. \par

\begin{lem} \label{lem:curvature_2d}
	For every $k\geq 0$ and every $k'>k$, there exists $\delta> 0$ with the following property. If $||B_L||\leq k$ and $L'=\operatorname{graph}\xi$ with $||\xi||,||\xi'||<\delta$, then $||B_\alpha||:=||B_{\operatorname{graph}\xi_\alpha}||\leq \max\{k',||B_{L'}||\}$ for all $\alpha\in [0,1]$.
\end{lem}

The lemma itself relies on the following computation. \par

\begin{lem} \label{lem:computation-curvature-graph_2d}
	The geodesic curvature of $\operatorname{graph}\xi$ at a point is given by
	\begin{align*}
		|B|=\frac{|W|}{\left(|W|^2+|\xi'|^2\right)^{\frac{3}{2}}}\left|\xi''+\frac{1}{2}\frac{\del}{\del t}|W|^2-\frac{\xi'}{|W|^2}\left(\frac{1}{2}\frac{\del}{\del s}|W|^2+\xi'\frac{\del}{\del t}|W|^2\right)\right|,
	\end{align*}
	where $|W|(s):=|W(s,\xi(s))|$.
\end{lem}
\begin{pr*}
	This is a direct computation, but we give here the important steps. The only nonzero Christoffel symbols of $\phi^*g$ are
	\begin{align*}
		\Gamma_{ss}^s=\frac{1}{2|W|^2}\frac{\del}{\del s}|W|^2, \quad \Gamma_{st}^s=\Gamma_{ts}^s=\frac{1}{2|W|^2}\frac{\del}{\del t}|W|^2, \quad \Gamma_{ss}^t=\frac{1}{2}\frac{\del}{\del t}|W|^2.
	\end{align*}
	Therefore, if we set $\Gamma(s):=(s,\xi(s))$, we get that
	\begin{align*}
		\dot{\Gamma}&=\frac{\del}{\del s}+\xi'\frac{\del}{\del t}, \\
		J\dot{\Gamma}&=|W|\frac{\del}{\del t}-\frac{\xi'}{|W|}\frac{\del}{\del s},
		\intertext{and}
		\ddot{\Gamma}&=\frac{1}{|W|^2}\left(\frac{1}{2}\frac{\del}{\del s}|W|^2+\xi'\frac{\del}{\del t}|W|^2\right)\frac{\del}{\del s}+\left(\xi''+\frac{1}{2}\frac{\del}{\del t}|W|^2\right)\frac{\del}{\del t}.
	\end{align*}
	Since the geodesic curvature is given by
	\begin{align*}
		|B|=\frac{\left|\left\langle\ddot{\Gamma},J\dot{\Gamma}\right\rangle\right|}{|\dot{\Gamma}|^3},
	\end{align*}
	this gives the above formula.	
\end{pr*}

\begin{proof}[Proof of Lemma~\ref{lem:curvature_2d}]
	The proof is somewhat tedious but quite elementary. We first combine Lemmata~\ref{lem:taylor_jacobi_2d} and~\ref{lem:computation-curvature-graph_2d} to get the identity
	\begin{align} \label{eqn:taylor_curvature_2d}
		|B|^2=(1+R_1)\left(\xi''-\kappa+R_2\right)^2,
	\end{align}
	where $R_1$ and $R_2$ are smooth functions of $s$. Furthermore, at a given $s$, they only depend on $L$, $M$, $g$, and the values $\xi(s)$ and $\xi'(s)$, and in such a manner that $R_i\to 0$ as $\xi(s),\xi'(s)\to 0$. \par
	
	Note that (\ref{eqn:taylor_curvature_2d}) implies that if $||B||\leq k$ and $||\xi||,||\xi'||\leq \frac{r}{2}$, then $||\xi''||$ is bounded by some constant $C$ depending only on $L$, $M$, $g$, and the constants $k$ and $r$. This allows us to write
	\begin{align} \label{eqn:taylor-better_curvature_2d}
		|B_\alpha|^2=\left(\alpha\xi''-\kappa\right)^2+\alpha\cdot\mathcal{O}(|\xi|,|\xi'|),
	\end{align}
	for $C^1$-close graphs with bounded $\xi''$-dependence in the error term. That is to say that $|B_{\operatorname{graph}\xi^i_\alpha}|^2-\left(\alpha(\xi^i)''-\kappa\right)^2$ tends uniformly to 0 if $\{\xi^i\}$ $C^1$-converges to 0. Here, we have made use of the fact that $\xi''_\alpha=\alpha\xi''$, $\xi'_\alpha=\alpha\xi'$, and $|\xi_\alpha|\leq 2\alpha|\xi|$~---~this is part of the hypotheses on $c$. The rest of the proof then consists of studying the behaviour of the parabola $\alpha\mapsto (\alpha\xi''-\kappa)^2$ under small perturbations. \par
	
	Fix $\epsilon\in (0,\frac{1}{2})$. We break down the analysis into a few subcases.
	\begin{enumerate}[label=(\Alph*)]
		\item ($|\xi''|\leq\epsilon$): Then, $(\alpha\xi''-\kappa)^2\leq (|\kappa|+\epsilon)^2\leq (k+\epsilon)^2$. By supposing $|\xi|$ and $|\xi'|$ small enough, we may suppose the error term in (\ref{eqn:taylor-better_curvature_2d}) to be smaller than $\epsilon$, so that
		\begin{align*}
			|B_\alpha|^2\leq (k+\epsilon)^2+\epsilon.
		\end{align*}
		\item ($|\xi''|>\epsilon$): There are three subcases (see Figure~\ref{fig:cases_parabola} for a visualization).
		\begin{enumerate}[label=(\alph*)]
			\item ($\frac{2\kappa}{\xi''}+\epsilon\geq 1$): In this case, $\frac{\kappa}{\xi''}>0$, and we have that the maximum of $\alpha\mapsto (\alpha\xi''-\kappa)^2$ on the interval $[0,\frac{2\kappa}{\xi''}]$ is reached at $\alpha=0$. Therefore, its maximum value on $[0,\frac{2\kappa}{\xi''}]$ is $\kappa^2\leq k^2$. \par
			
			On $[\frac{2\kappa}{\xi''},\frac{2\kappa}{\xi''}+\epsilon]$, the parabola is increasing with derivative
			\begin{align*}
				2(\alpha\xi''-\kappa)\xi''=2|\alpha\xi''-\kappa|\ |\xi''|\leq 2 \left(\alpha|\xi''|+|\kappa|\right)|\xi''|\leq 2C(C+k).
			\end{align*}
			Therefore, the maximum value of the parabola over that subinterval is bounded from above by $k^2+2C(C+k)\epsilon$. Thus, for $|\xi|$ and $|\xi'|$ as in (A), we get
			\begin{align*}
				|B_\alpha|^2\leq k^2+(2C(C+k)+1)\epsilon
			\end{align*}
			since $[0,1]\subseteq [0,\frac{2\kappa}{\xi''}+\epsilon]$.
			
			\item ($\frac{2\kappa}{\xi''}+\epsilon\leq 0$): In this case, the parabola is increasing over $[0,1]$ with derivative
			\begin{align*}
				2(\alpha\xi''-\kappa)\xi''=2|\alpha\xi''-\kappa|\ |\xi''|\geq  2|\kappa|\ |\xi''|\geq \epsilon|\xi''|^2 > \epsilon^3.
			\end{align*}
			By supposing $|\xi|$ and $|\xi'|$ small enough, we may suppose the $\alpha$-derivative of the error term in (\ref{eqn:taylor-better_curvature_2d}) to be smaller than $\frac{\epsilon^3}{2}$. Therefore, $\alpha\mapsto |B_\alpha|$ is still increasing, and we have that
			\begin{align*}
				|B_\alpha|\leq |B_1|=|B_{L'}|.
			\end{align*}
			
			\item ($0<\frac{2\kappa}{\xi''}+\epsilon< 1$): This case combines the approaches of (a) and (b). Indeed, the estimates of (a) still hold for $\alpha\in [0,\frac{2\kappa}{\xi''}+\epsilon]\subseteq [0,1]$. Likewise, over the interval $[\frac{2\kappa}{\xi''}+\epsilon,1]$, the parabola is increasing with derivative
			\begin{align*}
				2(\alpha\xi''-\kappa)\xi''=2|\alpha\xi''-\kappa|\ |\xi''|\geq 2|\epsilon\xi''+\kappa||\xi''|>\epsilon|\xi''|^2>\epsilon^3.
			\end{align*}
			Therefore, for $|\xi|$ and $|\xi'|$ as in (b), we have that
			\begin{align*}
				|B_\alpha|\leq \max\{k^2+(2C(C+k)+1)\epsilon,|B_{L'}|\}.
			\end{align*}
		\end{enumerate}
	\end{enumerate}
	
	\begin{figure}[ht]
		\centering
		\begin{tikzpicture}
			\node[anchor=south west,inner sep=0] (image) at (0,0) {\includegraphics[width=0.27\textwidth]{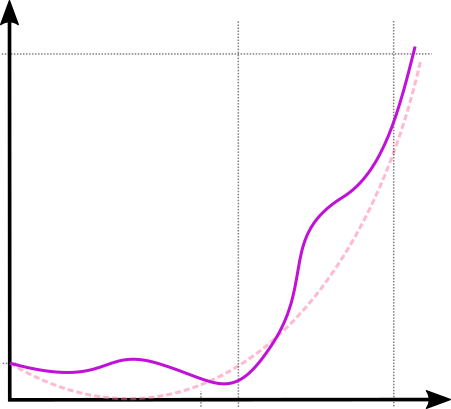}};
			\begin{scope}[x={(image.south east)},y={(image.north west)}]
				\node[] at (0.44,-0.06) {{\scriptsize $\frac{1}{2}$}};
				\node[] at (0.54,-0.06) {{\scriptsize $\frac{2\kappa}{\xi''}$}};
				\node[] at (0.875,-0.06) {{\scriptsize $1$}};
				\node[] at (-0.045,0.14) {{\scriptsize $\kappa^2$}};
				\node[] at (-0.045,0.87) {{\scriptsize $K$}};
				\node[] at (0.5,-0.3) {Case (a)};
			\end{scope}
		\end{tikzpicture}
		\quad
		\begin{tikzpicture}
			\node[anchor=south west,inner sep=0] (image) at (0,0) {\includegraphics[width=0.27\textwidth]{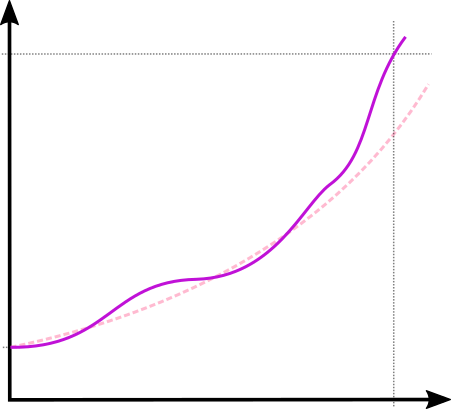}};
			\begin{scope}[x={(image.south east)},y={(image.north west)}]
				\node[] at (0.875,-0.06) {{\scriptsize $1$}};
				\node[] at (-0.045,0.18) {{\scriptsize $\kappa^2$}};
				\node[] at (-0.045,0.87) {{\scriptsize $K$}};
				\node[] at (0.5,-0.3) {Case (b)};
			\end{scope}
		\end{tikzpicture}
		\quad
		\begin{tikzpicture}
			\node[anchor=south west,inner sep=0] (image) at (0,0) {\includegraphics[width=0.27\textwidth]{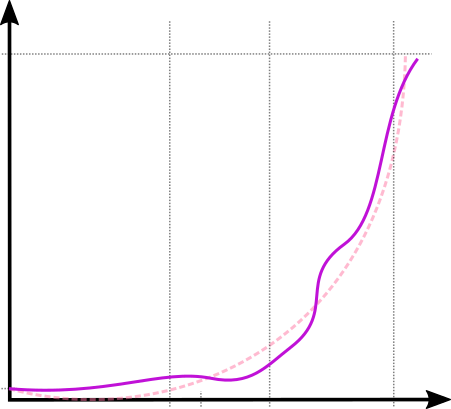}};
			\begin{scope}[x={(image.south east)},y={(image.north west)}]
				\node[] at (0.37,-0.06) {{\scriptsize $\frac{2\kappa}{\xi''}$}};
				\node[] at (0.455,-0.06) {{\scriptsize $\frac{1}{2}$}};
				\node[] at (0.63,-0.06) {{\scriptsize $\frac{2\kappa}{\xi''}+\epsilon$}};
				\node[] at (0.875,-0.06) {{\scriptsize $1$}};
				\node[] at (-0.045,0.08) {{\scriptsize $\kappa^2$}};
				\node[] at (-0.045,0.87) {{\scriptsize $K$}};
				\node[] at (0.5,-0.3) {Case (c)};
			\end{scope}
		\end{tikzpicture}
		\vspace*{8pt}
		\caption{Possible graph of $|B_\alpha|$ (full, purple) with idealized parabola (dashed, pink) and important values (pointed, grey). $K$ denotes whichever bound we get in the corresponding case. \label{fig:cases_parabola}}
	\end{figure}
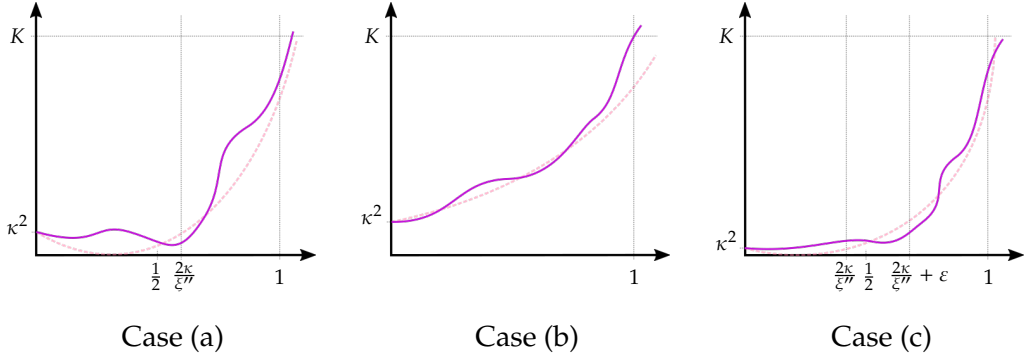
	
	The result then follows by taking $\epsilon$ such that
	\begin{align*}
		\max\{(k+\epsilon)^2+\epsilon, k^2+(2C(C+k)+1)\epsilon\}\leq (k')^2.
	\end{align*}
\end{proof}

We now move on to the analysis of the tameness constant of $\operatorname{graph}\xi_\alpha$. The approach that we take here is similar in spirit to the one taken in the proof of Lemma~\ref{lem:tame_TL}. \par

\begin{lem} \label{lem:tameness_2d}
	Given $L'=\operatorname{graph}\xi\subseteq \phi([0,\ell]\times (-r,r))$, denote by $d_\xi$ the distance function on $\operatorname{graph}\xi$ induced by the Riemannian metric $\phi^*g|_{\operatorname{graph}\xi}$. Take
	\begin{align*}
		\epsilon_\xi:=\inf_{x\neq x'\in \operatorname{graph}\xi}\frac{d_M(x,x')}{\min\{1,d_\xi(x,x')\}}\in (0,1].
	\end{align*}
	We have that $\epsilon_\xi\to\epsilon_0$ as $\xi\xrightarrow{C^1}0$.
\end{lem}
\begin{pr*}
	Fix $x\neq x'\in \operatorname{graph}\xi$, and take $\Gamma(s)=(s,\xi(s))$. Without loss of generality, we may suppose that $x=\Gamma(s_0)$ and $x'=\Gamma(s_1)$ for $0<s_1-s_0\leq\frac{\ell}{2}$. Therefore, $d_\xi(x,x')$ is simply the length $\ell(\Gamma)$ of $\Gamma$ along the interval $[s_0,s_1]$. Note that
	\begin{align*}
		|\dot{\Gamma}|^2=|W|^2+|\xi'|^2.
	\end{align*}
	By Lemma~\ref{lem:taylor_jacobi_2d}, we thus have that
	\begin{align*}
		d_\xi(x,x')\leq \int_{s_0}^{s_1} \left(1+C(||\xi||+||\xi'||)\right)ds = \left(1+C(||\xi||+||\xi'||)\right)(s_1-s_0)
	\end{align*}
	for some constant $C\geq 0$ depending only on $M$, $L$, $g$, and $r$. \par
	
	On the other hand, the path $t\mapsto (s,t\xi(s))$, $t\in [0,1]$, has length $|\xi(s)|$ for any $s$. Therefore, we get that
	\begin{align*}
		d_M(x,x')\geq d_M(\gamma(s_0),\gamma(s_1))-d_M(\gamma(s_0),x)-d_M(x',\gamma(s_1))\geq d_M(\gamma(s_0),\gamma(s_1))-2||\xi||.
	\end{align*}
	
	Combining these two estimates, we get that
	\begin{align*}
		\frac{d_M(\Gamma(s_0),\Gamma(s_1))}{\min\{1,d_\xi(\Gamma(s_0),\Gamma(s_1))\}}\geq \frac{d_M(\gamma(s_0),\gamma(s_1))-2||\xi||}{\left(1+C(||\xi||+||\xi'||)\right)\min\{1,d_L(\gamma(s_0),\gamma(s_1))\}},
	\end{align*}
	since $s_1-s_0=d_L(\gamma(s_0),\gamma(s_1))$. We similarly get
	\begin{align*}
		\frac{d_M(\Gamma(s_0),\Gamma(s_1))}{\min\{1,d_\xi(\Gamma(s_0),\Gamma(s_1))\}}\leq \frac{d_M(\gamma(s_0),\gamma(s_1))+2||\xi||}{\left(1-C(||\xi||+||\xi'||)\right)\min\{1,d_L(\gamma(s_0),\gamma(s_1))\}}.
	\end{align*}
	However, note that both functions on the right hand side uniformly~---~in $s_0$ and $s_1$~---~converge to the ratio $d_M/\min\{1,d_L\}$. Therefore, the same holds for the left hand side. Since uniform limits and infima commute, the result ensues.
\end{pr*}

\subsection{Comparison between metrics} \label{subsec:geo_comparison}
To simplify our computations, we only prove our comparison theorems between metrics of the form $g=\omega(-,J-)$, where $J$ is an $\omega$-compatible almost complex structure and $g$ is complete. Thus, in what follows, $g$ and $g'$ are two such metrics coming from the same symplectic form $\omega$. We however expect analogous results to hold in full generality. \par

\begin{lem} \label{lem:curvature_comparision}
	Let $W\subseteq M$ be compact. There are $a\geq 1$ and $b\geq 0$ such that every Lagrangian $L\subseteq W$ with $||B_L||<k$ respects $||B'_L||'<ak+b$, where $||B'||'$ is the norm in $g'$ of the second fundamental form of $L$ in $g'$.
\end{lem}
\begin{pr*}
	Let $A$ be the endomorphism of $TM$ such that $g(AV,W)=g'(V,W)$ for all $V,W\in\mathfrak{X}(M)$. We denote by $\nabla$ and $\nabla'$ the Levi-Civita connections of $g$ and $g'$, respectively. A fairly straightforward computation~---~see Lemma 3.1 of~\cite{GromanSolomon2016} for example~---~gives
	\begin{align*}
		2g'(\nabla'_V W,Z)
		&=2g(\nabla_V W,AZ)+g((\nabla_V A)Z,W) \\
		&\qquad +g((\nabla_W A)Z,V)-g((\nabla_Z A)V,W)
	\end{align*}
	for all $V,W,Z\in\mathfrak{X}(M)$. But it is easy to see that $A=-J J'$. Therefore, we get that
	\begin{align*}
		|V|^2
		&=g'(A^{-1}V,V)\leq |A^{-1}|'\cdot(|V|')^2=|J|'\cdot(|V|')^2
		\intertext{and}
		|AZ|^2
		&=|J'Z|^2\leq |J|'\cdot(|J'Z|')^2=|J|'\cdot(|Z|')^2
	\end{align*}
	for all $V,Z\in\mathfrak{X}(M)$. Here, the prime on the norm indicates that it is the norm associated with $g'$, not $g$. Finally, note also that $A$ sends the normal bundle of a Lagrangian submanifold $L$ in $g$ to the normal bundle of that Lagrangian submanifold in $g'$. \par
	
	Putting all this together, we finally get, for any Lagrangian submanifold $L$, 
	\begin{align}
		||B'_L||'
		&\leq ||J||'^{\frac{3}{2}}\left(||B_L||+\frac{3}{2}||\nabla A^{-1}||\right),
	\end{align}
	which implies the result. The norm symbols $||\cdot||$ and $||\cdot||'$ here indicate that it is the supremum of the norms associated with $g$ and $g'$, respectively, over all of $W$ or all of $L$.
\end{pr*}

\begin{lem} \label{lem:tameness_comparision}
	Let $W\subseteq M$ be compact. There is $a\geq 1$ such that every submanifold $N\subseteq W$ which is (strictly) $k^{-1}$-tame in $g$ is (strictly) $a^{-1}k^{-1}$-tame in $g'$.
\end{lem}
\begin{pr*}
	Take $D:=\frac{1}{2}\mathrm{Diam}(W,g)$, $D'=\frac{1}{2}\mathrm{Diam}(W,g')$ and $V:=\overline{B_D(W)}\cup \overline{B'_{D'}(W)}$, i.e.\ $V$ is the union of the (closure of the) $D$-neighbourhood of $W$ in $g$ and $D'$-neighbourhood of $W$ in $g'$. Therefore, any path between points in $W$ leaving $V$ have length greater than $\mathrm{Diam}(W,g)$ in $g$ and than $\mathrm{Diam}(W,g')$ in $g'$. In particular, no such path is a minimal geodesic in either metric. Since $V$ is compact, there is $C\geq 1$ such that
	\begin{align*}
		C^{-1}g'\leq g\leq Cg'
	\end{align*}
	on $V$. \par
	
	Let $x,y\in W$, and let $\gamma$ be a minimal geodesic in $g'$ from $x$ to $y$. By the above paragraph, it stays in $V$. Therefore,
	\begin{align*}
		d'_M(x,y)=\ell'(\gamma)\geq C^{-1}\ell(\gamma)\geq C^{-1}d_M(x,y),
	\end{align*}
	where $d_M$ and $d'_M$ are the distance functions induced on $M$ by $g$ and $g'$, respectively, and $\ell(\gamma)$ and $\ell'(\gamma)$ denote the lengths of $\gamma$ in $g$ and $g'$, respectively. We analogously get that $d'_M\leq Cd_M$ on $W$. \par
	
	But note that whenever $N\subseteq W$, the same argument gives that $C^{-1}d'_N\leq d_N\leq d'_N$ if $d_N$ and $d'_N$ the distance functions induced on $N$ by $g|_{TN}$ and $g'|_{TN}$, respectively. Therefore, we have that
	\begin{align*}
		\frac{d'_M(x,y)}{\min\{1,d'_N(x,y)\}}\geq \frac{C^{-1}d_M(x,y)}{\min\{1,Cd_N(x,y)\}}\geq C^{-2} \frac{d_M(x,y)}{\min\{1,d_N(x,y)\}}
	\end{align*}
	for all $x\neq y\in N$, which gives the result.
\end{pr*}

\newpage
\bibliographystyle{alpha}
\bibliography{completion}

\Addresses
\end{document}